\pdfoutput=1
\RequirePackage{ifpdf}
\ifpdf 
\documentclass[pdftex]{sigma}
\else
\documentclass{sigma}
\fi

\numberwithin{equation}{section}

\newcommand{\D}{{\mathbb D}}
\newcommand{\C}{{\mathbb C}}
\newcommand{\N}{{\mathbb N}}
\newcommand{\Z}{{\mathbb Z}}
\newcommand{\K}{{\mathbb K}}

\newtheorem{theo}{Theorem}[section]
\newtheorem{cor}[theo]{Corollary}
\newtheorem{lem}[theo]{Lemma}
\newtheorem{prop}[theo]{Proposition}
 { \theoremstyle{definition}

\newtheorem{Example}[theo]{Example}
\newtheorem{defin}[theo]{Definition}}

\begin{document}

\newcommand{\arXivNumber}{1705.04005}

\renewcommand{\PaperNumber}{075}

\FirstPageHeading

\ShortArticleName{Derivations and Spectral Triples on Quantum Domains I: Quantum Disk}

\ArticleName{Derivations and Spectral Triples\\ on Quantum Domains I: Quantum Disk}

\Author{Slawomir KLIMEK~$^\dag$, Matt MCBRIDE~$^\ddag$, Sumedha RATHNAYAKE~$^\S$, Kaoru SAKAI~$^\dag$\\ and Honglin WANG~$^\dag$}

\AuthorNameForHeading{S.~Klimek, M.~McBride, S.~Rathnayake, K.~Sakai and H.~Wang}

\Address{$^\dag$~Department of Mathematical Sciences, Indiana University-Purdue University Indianapolis, \\
\hphantom{$^\dag$}~402 N.~Blackford St., Indianapolis, IN 46202, USA}
\EmailD{\href{mailto:sklimek@math.iupui.edu}{sklimek@math.iupui.edu}, \href{mailto:ksakai@iupui.edu}{ksakai@iupui.edu}, \href{mailto:wanghong@imail.iu.edu}{wanghong@imail.iu.edu}}

\Address{$^\ddag$~Department of Mathematics and Statistics, Mississippi State University,\\
\hphantom{$^\ddag$}~175 President's Cir., Mississippi State, MS 39762, USA}
\EmailD{\href{mailto:mmcbride@math.msstate.edu}{mmcbride@math.msstate.edu}}

\Address{$^S$~Department of Mathematics, University of Michigan, \\
\hphantom{$^S$}~530 Church St., Ann Arbor, MI 48109, USA}
\EmailD{\href{mailto:sumedhar@umich.edu}{sumedhar@umich.edu}}

\ArticleDates{Received May 12, 2017, in f\/inal form September 21, 2017; Published online September 24, 2017}

\Abstract{We study unbounded invariant and covariant derivations on the quantum disk. In particular we answer the question whether such derivations come from operators with compact parametrices and thus can be used to def\/ine spectral triples.}


\Classification{46L87; 46L89; 58B34; 58J42}

\section{Introduction}
Derivations in Banach algebras have been intensively studied, originally inspired by applications in quantum statistical mechanics. Problems such as classif\/ication, generator properties, closedness of domains were the focus of the attention. Good overviews are \cite{B} and \cite{S}. More recently derivations were studied in connection with the concept of noncommutative vector f\/ields, partially inspired by Connes work \cite{Connes}.

An abstract def\/inition of a f\/irst-order elliptic operator is given by the concept of a spectral triple. A spectral triple is a triple $(\mathcal A, \mathcal H, \mathcal D)$ where $\mathcal H$ is a Hilbert space on which a $C^*$-algebra $A$ acts by bounded operators, $\mathcal A$ is a dense $^*$-subalgebra of $A$, and $\mathcal D$ is an unbounded self-adjoint operator in $\mathcal H$ such that $[\mathcal D, a]$ is bounded for
$a \in \mathcal A$, and $(I+ \mathcal D^2)^{-1/2}$ is a compact operator.

It is therefore natural to look at a situation where the commutator $[\mathcal D, a]$ is not just bounded but belongs to the algebra
$A$ in $B(\mathcal H)$, i.e., when it is an unbounded derivation of $A$ with domain~$\mathcal A$. The question is then about the compactness of the resolvent.

If $(\mathcal A, \mathcal H, \mathcal D)$ is an even spectral triple then ${\mathcal D}$ is of the form \begin{gather*}
{\mathcal D} = \left[
\begin{matrix}
0 & D \\
D^* & 0
\end{matrix}\right]
\end{gather*}
for a closed operator $D$. Then the spectral triple conditions require compactness of both $(I+D^*D)^{-1}$ and $(I+DD^*)^{-1}$. Those conditions are equivalent to saying that $D$ has {\it compact parametrices}, i.e., there are compact operators $Q_1$ and $Q_2$ such that $Q_1D-I$ and $DQ_2-I$ are compact, see the appendix.

A good example is the irrational rotation algebra, i.e., the noncommutative two-torus, def\/ined as the universal $C^*$-algebra $A_\phi$ with two unitary generators $u$ and $v$ such that $vu = e^{2\pi i\phi}uv$. It has two natural derivations $d_1$, $d_2$, def\/ined on the subalgebra $\mathcal A_\phi$ of polynomials in $u$, $v$ and its adjoints, by the following formulas on generators of~$A_\phi$
\begin{gather*}
d_1(u)=u,\qquad d_1(v)=0,\\
d_2(u)=0,\qquad d_2(v)=v.
\end{gather*}
Those derivations are generators of the torus action on $A_\phi$. In fact, according to \cite{BEJ}, any derivation $d\colon \mathcal A_\phi\to A_\phi$ can be uniquely decomposed into a linear combination of $d_1$, $d_2$ (invariant part) and an approximately inner derivation. The algebra $A_\phi$ has a natural representation in the GNS Hilbert space $L^2(A_\phi)$ with respect to the unique tracial state on $A_\phi$. Then, as described for example in~\cite{Connes}, the combination $D=d_1+id_2$ is implemented in the Hilbert space $L^2(A_\phi)$ by an operator with compact parametrices and thus leads to the canonical even spectral triple for the noncommutative torus.

In this paper we look at unbounded invariant and covariant derivations on the quantum disk, the Toeplitz $C^*$-algebra of the unilateral shift $U$, which has a natural $S^1$ action given by the multiplication of the generator $U\mapsto e^{i\theta}U$. We f\/irst classify such derivations and then look at their implementations in various Hilbert spaces obtained from the GNS construction with respect to an invariant state. We answer the question when such implementations are operators with compact parametrices and thus can be used to def\/ine spectral triples. Surprisingly, no implementation of a covariant derivation in any GNS Hilbert space for a faithful normal invariant state has compact parametrices for a large class of reasonable boundary conditions. This is in contrast with classical analysis, described in the following section, where for a d-bar operator, which is a~covariant derivation on the unit disk, subject to APS-like boundary conditions, the parametrices are compact. Similar analysis for the quantum annulus is contained in the follow-up paper~\cite{KMB2}.

The paper is organized as follows. In Section~\ref{section2} we describe two commutative examples of the circle and the unit disk which provide motivation for the remainder of the paper. In Section~\ref{section3} we review the quantum unit disk. Section~\ref{section4} contains a classif\/ication of invariant and covariant derivations in the quantum disk. In Section~\ref{section5} we classify invariant states on the quantum disk and describe the corresponding GNS Hilbert spaces and representations, while in Section~\ref{section6} we compute the implementations of those derivations in the GNS Hilbert spaces of Section~\ref{section5}. In Section~\ref{section7} we analyze when those implementations have compact parametrices. Finally, in Appendix~\ref{appendixA}, we review some general results about operators with compact parametrices needed for the analysis in Section~\ref{section7}.

\section{Commutative examples}\label{section2}

The subject of this paper is derivations in operator algebras.

\begin{defin}
Let $A$ be a Banach algebra and let $\mathcal{A}$ be a dense subalgebra of $A$. A linear map $d\colon \mathcal{A}\to A$ is called a {\it derivation} if the Leibniz rule holds
\begin{gather*}
d(ab) = ad(b) + d(a)b
\end{gather*}
for all $a,b\in\mathcal{A}$.
\end{defin}

If $A$ is a $*$-algebra, $\mathcal{A}$ is a dense $*$-subalgebra of $A$ and if $d(a^*) = (d(a))^*$, then $d$ is called a~$*$-{\it derivation}.

\begin{defin}
Let $A$ be a Banach algebra and $\mathcal{A}$ be a dense subalgebra of $A$ such that $\mathcal{A}\subsetneqq A $ and $d$ is a derivation with domain $\mathcal{A}$. The derivation $d$ is called {\it closed} if whenever $a_n,a\in\mathcal{A}$, $a_n\to a$ and $d(a_n)\to b$, then we have $d(a) = b$. Moreover, $d$ is called {\it closable} if $a_n\to 0$ and $d(a_n)\to b$ implies $b=0$.
\end{defin}

Closable derivation $d$ can be extended (non-uniquely) to a closed derivation, the smallest of which is called the {\it closure} of~$d$ and denoted by $\overline{d}$. In the following we will describe in some detail two commutative examples that have some features of, and provide a motivation for, our main object of study, the noncommutative disk.

\begin{Example}
Let $A = C(S^1)$ be the $C^*$-algebra of continuous functions on the circle $S^1=\{e^{ix},\, x\in [0,2\pi)\}$. If $\mathcal{A}$ is the algebra of trigonometric polynomials then\begin{gather*}
(da)(x) = \frac{1}{i}\frac{da(x)}{dx}
\end{gather*}
is an example of an unbounded $*$-derivation that is closable.
\end{Example}

Let $\rho_\theta\colon A\to A$ be the one-parameter family of automorphisms of $A$ obtained from rotation $x\mapsto x+\theta$ on the circle. The map $\tau\colon A\to \C$ given by \begin{gather*}
\tau(a)=\frac{1}{2\pi}\int_0^{2\pi}a(x)dx
\end{gather*}
is the unique $\rho_\theta$-invariant state on $A$ and, up to a constant, $d$ is the unique $\rho_\theta$-invariant derivation on $\mathcal{A}$. The Hilbert space $H_\tau$, obtained by the Gelfand--Naimark--Segal (GNS) construction on~$A$ using the state $\tau$, is naturally identif\/ied with $L^2(S^1)$, the completion of $A$ with respect to the usual inner product
\begin{gather*} ||a||^2_\tau=\tau(a^*a)=\frac{1}{2\pi}\int_0^{2\pi}|a(x)|^2dx.\end{gather*}
The representation $\pi_\tau\colon A\to B(H_\tau)$ is given by multiplication: $\pi_\tau(a)f(x)=a(x)f(x)$. Then the operator\begin{gather*}
(D_{\tau} f)(x) = \frac{1}{i}\frac{df(x)}{dx}
\end{gather*}
on domain $\mathcal{D}_{\tau}=\mathcal{A}\subset A\subset H_\tau$ is an {\it implementation} of $d$ in $H_\tau$ because of the relation\begin{gather*}
 [D_{\tau}, \pi_\tau(a)] = \pi_\tau(d(a)),
\end{gather*}
for $a\in\mathcal{A}$. The operator $D_{\tau}$ is rotationally invariant and has compact parametrices because its spectrum is $\Z$ and thus $(\mathcal{A}, H_\tau, D_{\tau})$ is a spectral triple.

\begin{Example}
The second example is the d-bar operator on the unit disk, and it is the motivating example for the rest of the paper.
Let $A = C(\D)$ be the $C^*$-algebra of continuous functions on the disk $\D=\{z\in\C\colon |z|\leq 1\}$. If $\mathcal{A}$ is the algebra of polynomials in~$z$ and $\bar z$ then\begin{gather*}
(da)(z) = \frac{\partial a(z)}{\partial \bar z}
\end{gather*}
is an unbounded, closable derivation in $A$.
\end{Example}

Let $\rho_\theta\colon A\to A$ be the one-parameter family of automorphisms of $A$ given by the rotation $z\to e^{i\theta}z$ on the disk. Notice that $\rho_\theta\colon \mathcal{A}\to\mathcal{A}$. Moreover, $d$ is a~{\it covariant} derivation in $A$ in the sense that it satisf\/ies\begin{gather*}
d(\rho_\theta(a))= e^{-i\theta}\rho_\theta(d(a)),\qquad a\in\mathcal{A}.
\end{gather*}

The map $\tau\colon A\to \C$ given by \begin{gather*}
\tau(a)=\frac{1}{\pi}\int_\D a(z)d^2z,
\end{gather*}
is a $\rho_\theta$-invariant, faithful state on $A$. The GNS Hilbert space $H_\tau$, obtained using the state~$\tau$, is naturally identif\/ied with $L^2(\D,d^2z)$, the completion of $A$ with respect to the usual inner product
\begin{gather*}||a||^2_\tau=\tau(a^*a)=\frac{1}{\pi}\int_\D|a(z)|^2d^2z.\end{gather*}
The representation $\pi_\tau\colon A\to B(H_\tau)$ is given by multiplication: $\pi_\tau(a)f(z)=a(z)f(z)$.
The unitary operator $U_{\tau,\theta}f(z):=f(e^{i\theta}z)$ in $H_\tau$ implements $\rho_\theta$ in the sense that\begin{gather*}
U_{\tau,\theta}\pi_\tau(a)U_{\tau,\theta}^{-1} = \pi_{\tau}(\rho_\theta(a)).
\end{gather*}
Then the covariant operator\begin{gather*}
(D_{\tau} f)(z) = \frac{\partial f(z)}{\partial\bar z}
\end{gather*}
on domain $\mathcal{D}_{\tau}=\mathcal{A}\subset A\subset H_\tau$ is an implementation of $d$ in $H_\tau$, i.e.,
$ [D_{\tau}, \pi_\tau(a)] = \pi_\tau(d(a))$, for all $a\in\mathcal{A}$. The operator $D_{\tau}$ however has an inf\/inite-dimensional kernel, so $(I+D_{\tau}^*D_{\tau})^{-1/2}$ is not compact. This is not a surprise; the disk is a manifold with boundary and we need to impose elliptic type boundary conditions to make~$D_{\tau}$ elliptic, so that it has compact parametrices.

Denote by $D^{\max}_\tau$ the closure of $D_{\tau}$, since there are no boundary conditions on its domain. On the other hand, let $D^{\min}_\tau$ be the closure of $D_{\tau}$ def\/ined on $C_0^\infty(\D)$. While $D^{\min}_\tau$ has no kernel, its cokernel now has inf\/inite dimension. The question then becomes of the existence of a closed operator $D_{\tau}$ with compact parametrices, such that $D^{\min}_\tau\subset D_{\tau}\subset D^{\max}_\tau$; this is answered in positive by Atiyah--Patodi--Singer (APS) type boundary conditions, see~\cite{BBW}. Spectral triples for manifolds with boundary using operators with APS boundary conditions were constructed in~\cite{BS}. References~\cite{C} and~\cite{DD} contain constructions of spectral triples for the quantum disk. Recent general framework for studying spectral triples on noncommutative manifolds with boundary is discussed in~\cite{FGMR}.

\section{Quantum disk}\label{section3}

Let $\{E_k\}$ be the canonical basis for $\ell^2(\N )$, with $\N$ being the set of nonnegative integers, and let $U$ be the unilateral shift, i.e., $UE_k = E_{k+1}$. Let $A$ be the $C^*$-algebra generated by $U$. The algebra~$A$ is called the Toeplitz algebra and by Coburn's theorem~\cite{Cob} it is the universal $C^*$-algebra with generator $U$ satisfying the relation $U^*U=I$, i.e., $U$ is an isometry. Reference \cite{KL} argues that this algebra can be thought of as a quantum unit disk. Its structure is described by the following short exact sequence
\begin{gather*}
0\longrightarrow \mathcal{K} \longrightarrow A \longrightarrow C\big(S^1\big) \longrightarrow 0,
\end{gather*}
where $ \mathcal{K} $ is the ideal of compact operators in $\ell^2(\N )$. In fact $\mathcal{K}$ is the commutator ideal of the algebra $A$.

We will use the diagonal label operator $\K E_{k} = kE_{k}$, so that, for a bounded function $a\colon \N \to\C$, we can write $a(\K) E_{k} = a(k)E_{k}$. We have the following useful commutation relation for a~diagonal operator~$a(\K)$
\begin{gather}\label{CommRel}
a(\K) U = Ua(\K +1).
\end{gather}

We call a function $a\colon \N \to\C$ {\it eventually constant}, if there exists a natural number $k_0$, called the {\it domain constant}, such that $a(k)$ is constant for $k\ge k_0$. The set of all such functions will be denoted by $c_{00}^+$. Let $\operatorname{Pol}(U, U^*)$ be the set of all polynomials in $U$ and $U^*$ and def\/ine
\begin{gather*}
\mathcal{A} = \bigg\{a = \sum_{n\geq 0}U^n a^+_n(\K)+\sum_{n\geq 1}a^-_n(\K)(U^*)^n \colon a^\pm_n(k)\in c_{00}^+,\ \text{f\/inite sums}\bigg\}.
\end{gather*}
We have the following observation.
\begin{prop}\label{pol_prop}
$\mathcal{A} = \operatorname{Pol}(U, U^*)$.
\end{prop}

\begin{proof}
Using the commutation relations~(\ref{CommRel}) it is easy to see that a product of two elements of~$\mathcal{A}$ and the adjoint of an element of $\mathcal{A}$ are still in~$\mathcal{A}$. It follows that $\mathcal{A}$ is a $*$-subalgebra of~$A$. Since~$U$ and $U^*$ are in~$\mathcal{A}$, it follows that $\operatorname{Pol}(U,V)\subset\mathcal{A}$. To prove the reverse inclusion it suf\/f\/ices to show that for any $a\in c_{00}^+$ the operator $a(\K)$ is in $\operatorname{Pol}(U,V)$, as the remaining parts of the sum are already polynomials in~$U$ and $U^*$. To show that $a(\K)\in \operatorname{Pol}(U,V)$, we decompose any $a(k)\in c_{00}^+$ in the following way
\begin{gather*}
a(\K) = \sum_{k=0}^{k_0-1}a(k)P_k + a_\infty P_{\ge k_0},
\end{gather*}
where $a_\infty = \lim\limits_{k\to\infty}a(k)$, $P_k$ is the orthogonal projection onto the one-dimensional subspace generated by~$E_k$ and $P_{\ge k_0}$ is the orthogonal projection onto $\operatorname{span}\{E_{k}\}_{k\ge k_0}$. A straightforward calculation shows that $U^k(U^*)^k = P_{\ge k}$ and $P_k = P_{\ge k} - P_{\ge k+1}$. This completes the proof.
\end{proof}

Let $c$ be the space of convergent sequences, and consider the algebra
\begin{gather*}A_{\rm diag} = \left\{a(\K) \colon \{a(k)\}\in c \right\}.\end{gather*}
This is precisely the subalgebra of all diagonal operators in $A$ and we can view the quantum disk as the semigroup crossed product of $A_{\rm diag}$ with $\N$ acting on $A_{\rm diag}$ via shifts (translation by $n\in\N$), that is \begin{gather*}A = A_{\rm diag} \rtimes_{\rm shift}\N.\end{gather*}
Several versions of the theory of semigroup crossed products exist, see for example \cite{St}.

\section{Derivations on quantum disk}\label{section4}

For each $\theta\in[0,2\pi)$, let $\rho_\theta \colon A\to A$ be an automorphism def\/ined by $\rho_\theta(U) = e^{i\theta}U$ and $\rho_\theta(U^*) = e^{-i\theta}U^*$. It is well def\/ined on all of $A$ because it preserves the relation $U^*U=I$. Alternatively, the action of $\rho_\theta$ can be written using the label operator $\K$ as\begin{gather*}
\rho_\theta(a)=e^{i\theta\K}ae^{-i\theta\K}.
\end{gather*}
It follows that $\rho_\theta(a(\K)) = a(\K)$ for a diagonal operator $a(\K)$ and $\rho_\theta \colon \mathcal{A}\to \mathcal{A}$.

Any derivation $d\colon \mathcal{A}\to A$ that satisf\/ies the relation $\rho_\theta(d(a)) = d(\rho_{\theta}(a))$ will be referred to as a $\rho_\theta$-{\it invariant} derivation. Similarly, any derivation $d\colon \mathcal{A}\to A$ that satisf\/ies the relation $d(\rho_\theta(a))= e^{-i\theta}\rho_\theta(d(a))$ for all $a\in A$ will be referred to as a $\rho_\theta$-{\it covariant} derivation.

Notice that, as a consequence of Proposition \ref{pol_prop}, we have the identif\/ications
\begin{gather*}\{a\in\mathcal{A}\colon \rho_\theta(a) = a\}= \left\{a(\K) \colon \{a(k)\}\in c_{00}^+ \right\}=: \mathcal{A}_{\rm diag},\end{gather*}
and similarly
\begin{gather*}\{a\in A \colon \rho_\theta(a) = a\}= A_{\rm diag}= \left\{a(\K) \colon \{a(k)\}\in c \right\}.\end{gather*}
We will also use the following terminology: we say that a function $\beta\colon \N \to\C$ has {\it convergent increments}, if the sequence of dif\/ferences $\{\beta(k)-\beta(k-1)\}$ is convergent, i.e., is in $c$. The set of all such functions will be denoted by $c_{\rm inc}$. Similarly the set
of {\it eventually linear} functions is the set of $\beta\colon \N \to\C$ such that $\{\beta(k)-\beta(k-1)\}\in c_{00}^+$.

The following two propositions classify all invariant and covariant derivations $d\colon \mathcal{A}\to A$.

\begin{prop}\label{invar_der_rep}
If $d$ is an invariant derivation $d\colon \mathcal{A}\to A$, then there exists a unique function $\beta\in c_{\rm inc}$, $\beta(-1)=0$, such that \begin{gather*}d(a) = [\beta(\K-1), a]\end{gather*} for $a\in\mathcal{A}$. If $d\colon \mathcal{A}\to \mathcal{A}$ then the corresponding
function $\beta(k)$ is eventually linear.
\end{prop}

\begin{proof}
Let $d(U^*) = f\in A$ and since $U^*U=I$ we get
\begin{gather*}
0 = d(I) = d(U^*U) = d(U^*)U + U^*d(U),
\end{gather*}
which implies that $U^*d(U) = -fU$. This in turn implies that $d(U) = -UfU + g$ for some $g\in A$ such that $U^*g=0$. Notice that $0=UU^*g=(1-P_0)g$.

Applying $\rho_\theta$ to $f$ we get the following\begin{gather*}
\rho_\theta(f) = \rho_\theta(d(U^*)) = d(\rho_\theta(U^*)) = e^{-i\theta}d(U^*) = e^{-i\theta}f .
\end{gather*}
A similar calculation shows that $\rho_\theta(g) = e^{i\theta}g$. Those covariance properties imply that $f=-\alpha(\K)U^*$ for some $\alpha(\K)\in A_{\rm diag}$ and
similarly $g=U\gamma(\K)$. However, since $g=P_0g$, and $P_0U=0$, we must have $g=0$.

Next, def\/ine $\beta\in c_{\rm inc}$ by $\beta(-1):=0$ and
\begin{gather*}\beta(k):=\sum_{j=0}^k\alpha(j).\end{gather*}
Then we have $\alpha(\K) = \beta(\K) - \beta(\K-1)$, and the result follows.
\end{proof}

The following description of covariant derivations is proved exactly the same as the proposition above.

\begin{prop}\label{covar_der_rep}
If $d$ is a covariant derivation on $\mathcal{A}$, then there exists a unique function $\beta\in c_{\rm inc}$, $\beta(-1):=0$, such that
\begin{gather*}d(a) = [U\beta(\K),a]\end{gather*}
for all $a\in\mathcal{A}$.
\end{prop}

Reference \cite{BEJ} brought up the question of decomposing derivations into approximately inner and invariant, not approximately inner parts, see also \cite{H,J}.
Below we study when invariant derivations in the quantum disk are approximately bounded/approximately inner. Recall that~$d$ is called {\it approximately inner} if there are $a_n\in A$ such that $d(a) = \lim\limits_{n\to\infty}[a_n,a]$ for $a\in\mathcal{A}$. If $d(a) = \lim\limits_{n\to\infty} d_n(a)$ for bounded derivations $d_n$ on $A$ then $d$ is called {\it approximately bounded}. Note also that any bounded derivation $d$ on $A$ can be written as a commutator $d(a)=[a,x]$ with $x$ in a weak closure of $A$; see \cite{KR,S}. In fact $x$ must belong to the essential commutant of the unilateral shift, which is not well understood~\cite{BH}.

\begin{lem}Let $d$ be a $\rho_\theta$-invariant derivation in $A$ with domain $\mathcal{A}$. If $d$ is approximately bounded then there exists a sequence $\{\mu_n(k)\}\in\ell^\infty$ such that
\begin{gather*}
d(a) = \lim_{n\to\infty}[a,\mu_n(\K-1)]
\end{gather*}
for all $a\in\mathcal{A}$.
\end{lem}

\begin{proof}Given an element $a\in A$ we def\/ine its $\rho_\theta$ average $a_{\rm av} \in A$ by\begin{gather*}
a_{\rm av} := \frac{1}{2\pi}\int_0^{2\pi} \rho_\theta(a) d\theta .
\end{gather*}
It follows that $a_{\rm av}$ is $\rho_\theta$-invariant since the Lebesgue measure $d\theta$ is translation invariant. Additionally, all $\rho_\theta$-invariant operators in $\ell^2(\N )$ are diagonal with respect to the basis $\{E_k\}$ so that $a_{\rm av}\in A_{\rm diag}$.

Since by assumption $d$ is approximately bounded, there exists a sequence of bounded operators~$b_n$ such that $d(a) = \lim\limits_{n\to\infty}[a,b_n]$ for all $a\in\mathcal{A}$. It suf\/f\/ices to show that
\begin{gather}\label{diag_conv}
\lim_{n\to\infty}[a,(b_n)_{\rm av}] = d(a),
\end{gather}
since $(b_n)_{\rm av}$ is $\rho_\theta$-invariant for every $\theta$ and hence by Proposition \ref{invar_der_rep} it is given by the commutator with a diagonal operator $\mu_n(\K-1)$ with the property $\{\mu(k)\}\in\ell^\infty$ because of the assumption of boundedness.

It is enough to verify \eqref{diag_conv} on the generators of the algebra $\mathcal{A}$; we show the calculation for $a=U$. We have, equivalently
\begin{gather*}b_n - U^*b_nU\to U^*d(U)\end{gather*}
as $n\to\infty$, and this means that for every $\varepsilon>0$ there exists $N$ such that for all $n>N$ we have
\begin{gather*}\|b_n-U^*b_nU-U^*d(U)\|<\varepsilon.\end{gather*}
So, because $U^*d(U)$ is $\rho_\theta$-invariant, and because
\begin{gather*}U^*\rho_\theta(b_n)U=\rho_\theta(U^*b_nU),\end{gather*}
we have
\begin{gather*}\|\rho_\theta(b_n)-U^*\rho_\theta(b_n)U-U^*d(U)\|=\|\rho_\theta\left(b_n-U^*b_nU-U^*d(U)\right)\|<\varepsilon,\end{gather*}
and thus we get the estimate
\begin{gather*}
\|(b_n)_{\rm av} - U^*(b_n)_{\rm av}U-U^*d(U)\|\le \frac{1}{2\pi}\int_0^{2\pi}\|\rho_\theta(b_n)-U^*\rho_\theta(b_n)U-U^*d(U)\|d \theta < \varepsilon.
\end{gather*}
This completes the proof.
\end{proof}

\begin{theo}
Let $d(a) = [\beta(\K-1), a]$ be a $\rho_\theta$-invariant derivation in $A$ with domain $\mathcal{A}$. If $d$ is approximately bounded then $\{\beta(k)-\beta(k-1)\}\in c_0$, the space of sequences converging to zero.
\end{theo}

\begin{proof}
By the previous lemma there exists $\{\mu_n(k)\}\in\ell^\infty$ such that
\begin{gather*}
d(a) = \lim_{n\to\infty} [a ,\mu_n(\K-1)]
\end{gather*}
for all $a\in\mathcal{A}$. Without loss of generality assume $\beta(k)$ and $\mu_n(k)$ are real, or else consider the real and imaginary parts separately. Suppose that $\{\beta(k)-\beta(k-1)\}\notin c_0$, then
\begin{gather*}
\lim_{k\to\infty}(\beta(k) - \beta(k-1)) = L\neq 0.
\end{gather*}
We can assume $L>0$; an identical argument works for $L<0$. The above equation implies that
\begin{gather*}
\lim_{n\to\infty}\underset{k}{\sup }|(\mu_n(k)-\mu_n(k-1)) - (\beta(k) - \beta(k-1))| = 0.
\end{gather*}
Therefore for $k$ and $n$ large enough we have
\begin{gather*}
L-\varepsilon \le \mu_n(k) - \mu_n(k-1)\le L+ \varepsilon,
\end{gather*}
and, by telescoping $\mu_n(k)$, we get
\begin{gather*}
\mu_n(k) = (\mu_n(k) - \mu_n(k-1)) + \cdots + (\mu_n(k_0) - \mu_n(k_0-1)) + \mu_n(k_0-1)
\end{gather*}
for some f\/ixed $k_0$. Together this implies that $\mu_n(k)\ge (L-\varepsilon)k + \mu_n(k_0-1)$ which goes to inf\/inity as $k$ goes to inf\/inity. This contradicts the fact that $\{\mu_n(k)\}\in\ell^\infty$ which ends the proof.
\end{proof}

We also have the following converse result.

\begin{theo}
If $d(a) = [\beta(\K-1), a]$ is a $\rho_\theta$-invariant derivation in $A$ with domain $\mathcal{A}$ such that $\{\beta(k) - \beta(k-1)\}\in c_0$, then $d$ is approximately inner.
\end{theo}

\begin{proof}
We show that there exists a sequence $\{\mu_n(k)\}\in c$ such that $[a,\mu_n(\K-1)]$ converges to $[a,\beta(\K-1)]$ for all $a\in\mathcal{A}$. As before, it is enough to verify this on the generators; we show the calculation for $a=U$. Thus we want to construct $\mu_n$ such that
\begin{gather*}
\lim_{n\to\infty} U^*[U,\mu_n(\K-1)] = U^*[U,\beta(\K-1)] .
\end{gather*}
The above equation is true if and only if the following is true
\begin{gather*}
\lim_{n\to\infty}(\mu_n(k)-\mu_n(k-1)) = \beta(k) - \beta(k-1).
\end{gather*}
The above in turn is true if and only if
\begin{gather*}
\underset{k}{\sup }\left|(\mu_n(k)-\mu_n(k-1)) - (\beta(k)-\beta(k-1))\right| \to 0 \qquad \text{as} \quad n\to\infty
\end{gather*}
is true. Def\/ine the sequence $\{\mu_n\}\in c$ by the following formulas
\begin{gather*}
\mu_n(k) =
\begin{cases}
\beta(k) &\text{for }k\le n, \\
\beta(n) &\text{for }k>n.
\end{cases}
\end{gather*}
It follows that for $k\le n$ we have $\mu_n(k) - \mu_n(k-1) = \beta(k)-\beta(k-1)$ and $\mu_n(k) - \mu_n(k-1) =0$ otherwise. Therefore we have
\begin{gather*}
\lim_{n\to\infty}\underset{k}{\sup }\left|(\mu_n(k)-\mu_n(k-1)) - (\beta(k)-\beta(k-1))\right| = \underset{k>n}{\sup }|\beta(k)-\beta(k-1)| = 0,
\end{gather*}
since $\{\beta(k)-\beta(k-1)\}\in c_0$. Thus the proof is complete.
\end{proof}

Notice that in the above theorem the derivation $d$ need not be bounded. For example, if $\beta(k) = \sqrt{k+1}$ then $\beta(k)-\beta(k-1)\to 0$ as $k\to\infty$, so, by the above theorem, $d$ is approximately inner. However, $d$ is unbounded.

\section{Invariant states}\label{section5}
Next we describe all the invariant states on $A$. If $\tau \colon A\to\C$ is a state, then $\tau$ is called a~$\rho_\theta$-invariant state on $A$ if it satisf\/ies $\tau(\rho_\theta(a)) = \tau(a)$.

Since $A = A_{\rm diag} \rtimes_{\rm shift}\N$, there is a natural expectation $E \colon A\to A_{\rm diag}$, i.e., $E$ is positive, unital and idempotent. For $a\in\mathcal{A}$ we have
\begin{gather}\label{E_expectation}
E(a) = E\bigg(\sum_{n\geq 0}U^n a^+_n(\K)+\sum_{n\geq 1}a^-_n(\K)(U^*)^n\bigg) = a_0(\K),
\end{gather}
and $a_0(\K)\in A_{\rm diag}$. Since $A_{\rm diag}$ is the f\/ixed point algebra for $\rho_\theta$, we immediately obtain the following lemma:

\begin{lem}Suppose $\tau \colon A\to\C$ is a $\rho_\theta$-invariant state on $A$. Then there exists a state $t \colon A_{\rm diag}\to\C$ such that $\tau(a) = t(E(a))$ where $E$ is the natural expectation. Conversely given the natural expectation $E$ and a state $t \colon A_{\rm diag}\to\C$, then $\tau(a) = t(E(a))$ defines a $\rho_\theta$-invariant state on $A$.
\end{lem}

To parametrize all invariant states we need to f\/irst identify the pure states.
\begin{lem}The pure states on $A_{\rm diag}$ denoted by $t_k$ for $k\in\N$ and $t_\infty$ are given by
\begin{gather*}
t_k(a(\K)) = a(k) = \langle E_k, aE_k\rangle, \\
t_\infty(a(\K)) = \lim_{k\to\infty} a(k) = \lim_{k\to\infty} t_k(a(\K)).
\end{gather*}
\end{lem}

\begin{proof}$A_{\rm diag}$ is a commutative $C^*$-algebra that is isomorphic to the algebra of continuous functions on the one-point compactif\/ication of $\N$, that is
\begin{gather*}
A_{\rm diag} \cong C(\N\cup\{\infty\}).
\end{gather*}
So by general theory, see \cite{KR} for details, the pure states are the Dirac measures (or point mass measures).
\end{proof}

As a consequence, we have the following classif\/ication theorem of the $\rho_\theta$-invariant states on~$A$.

\begin{theo}\label{state_decomp} The $\rho_\theta$-invariant states on $A$ are in the closed convex hull of the states~$\tau_k$ and~$\tau_\infty$ where $\tau_k(a) = t_k(E(a))$ and $\tau_\infty(a) = t_\infty(E(a))$. Explicitly, if $\tau$ is a~$\rho_\theta$-invariant state, there exist weights $w(k)\ge0$ such that $\sum\limits_{k\geq 0} w(k) = 1$ and non-negative numbers~$\lambda_0$ and~$\lambda_\infty$, with $\lambda_0 + \lambda_\infty = 1$ such that
\begin{gather*}
\tau = \lambda_\infty\tau_\infty + \lambda_0\sum_{k\geq 0}w(k)\tau_k.
\end{gather*}
In fact, we have $\sum_kw(k)\tau_k(a) = \operatorname{tr}(w(\K)a)=\tau_w(a)$, and $\lambda_0 = \sum_k\tau(P_k)$, $w(k) = \lambda_0^{-1}\tau(P_k)$, and $\lambda_\infty = 1-\sum\limits_{k\ge0}\tau(P_k)$ where again $P_k$ is the orthogonal projection onto the one-dimensional subspace spanned by $E_k$.
\end{theo}

\begin{proof}
By continuity it is enough to compute $\tau (a)$ on the dense set~$\mathcal{A}$. Then, by $\rho_\theta$-invariance and equation~(\ref{E_expectation}), we have
\begin{gather*}
\tau (a) = \tau\bigg(\sum_{n\geq 0}U^n a^+_n(\K)+\sum_{n\geq 1}a^-_n(\K)(U^*)^n\bigg) = \tau (a_0(\K) ).
\end{gather*}
Set $\tau(P_k) = \omega(k)$ and notice that $\omega(k)\ge0$ since $\tau(P_k) = \tau(P_k^2) = \tau(P_k^*P_k) \ge 0$. It is clear that $\omega(k)\le 1$ since~$P_k$ are projections. Next decompose any~$a(\K)\in\mathcal{A}$ as in the proof of Proposition~\ref{pol_prop}
\begin{gather*}
a(\K) = \sum_{k =0}^{L-1}a(k)P_k + a_\infty P_{\ge L},
\end{gather*}
where $L$ is the domain constant and $a_\infty$ is the value of $a(k)$ for $k\geq L$ . Applying $\tau$ to this decomposition we get
\begin{gather*}
\tau(a(\K)) =\sum_{k = 0}^{L-1}a(k)\omega(k) + a_\infty \tau(P_{\ge L}) = \sum_{k = 0}^{L-1} a(k)\omega(k) + a_\infty \tau(I - P_0 - P_1 -\cdots -P_{L-1}) \\
\hphantom{\tau(a(\K))}{} = \sum_{k = 0}^{L-1}a(k)\omega(k) + a_\infty \left(1 - \sum_{k=0}^{L-1} \omega(k)\right).
\end{gather*}
On the other hand we have
\begin{gather*}
\sum_{k\geq 0}a(k)\omega(k) = \sum_{k = 0}^{L-1}a(k)\omega(k) + a_\infty \sum_{k\ge L}\omega(k) .
\end{gather*}
Plugging this equation into the previous one we obtain
\begin{gather*}
\tau(a(\K)) = \sum_{k\geq 0} a(k)\omega(k) + a_\infty \left(1 - \sum_{k\ge 0}\omega(k)\right) =\sum_{j\in\N }\omega(j)\left(\sum_{k\geq 0}\frac{\omega(k)a(k)}{\sum_{j\in\N }\omega(j)}\right) \\
\hphantom{\tau(a(\K)) =}{} + a_\infty \left(1 - \sum_{k\ge 0}\omega(k)\right) = \lambda_0\left(\sum_{k\geq 0}\frac{\omega(k)a(k)}{\sum_{j\in\N }\omega(j)}\right) + a_\infty \lambda_\infty.
\end{gather*}
The last equation provides a convex combination of two states $\tau_\infty(a) = a_\infty$ and $\tau_w(a) = \operatorname{tr}(w(\K)a)$ with $w(k) = \frac{\omega(k)}{\sum_j \omega(j)}$ as
$\lambda_0 +\lambda_\infty = 1$. This completes the proof.
\end{proof}

Given a state $\tau$ on $A$ let $H_\tau$ be the GNS Hilbert space and let $\pi_\tau\colon A\to B(H_\tau)$ be the corresponding representation. We describe the three Hilbert spaces and the representations coming from the following three $\rho_\theta$-invariant states: $\tau_w$ with all $w(k)\ne 0$, $\tau_0$, and $\tau_\infty$. The states $\tau_w$ with all $w(k)\ne 0$ are general $\rho_\theta$-invariant faithful normal states on~$A$.

\begin{prop}\label{GNS_prop} The three GNS Hilbert spaces with respect to the $\rho_\theta$-invariant states $\tau_w$ with all $w(k)\ne 0$, $\tau_0$, and $\tau_\infty$ can be naturally identified with the following Hilbert spaces, respectively:
\begin{enumerate}\itemsep=0pt
\item[$1.$] $H_{\tau_w}$ is the Hilbert space whose elements are power series
\begin{gather*}f = \sum_{n\geq 0}U^n f^+_n(\K)+\sum_{n\geq 1}f^-_n(\K)(U^*)^n\end{gather*}
such that
\begin{gather}\label{w_inner_prod}
\|f\|_{\tau_w}^2 =\tau_w(f^*f)=\sum_{n\ge 0} \sum_{k=0}^\infty w(k)|f_n^+(k)|^2 + \sum_{n\ge1}\sum_{k=0}^\infty w(k+n)|f_n^-(k)|^2
\end{gather}
is finite.
\item[$2.$] $H_{\tau_0} \cong \ell^2(\N)$, $\pi_{\tau_0}(U)$ is the unilateral shift.
\item[$3.$] $H_{\tau_\infty} \cong L^2(S^1)$, $\pi_{\tau_\infty}(U)$ is the multiplication by $e^{ix}$.
\end{enumerate}
\end{prop}

\begin{proof}The f\/irst Hilbert space is just the completion of $A$ with respect to the inner product given by~\eqref{w_inner_prod}. It was discussed in~\cite{CKW}, and also~\cite{KM1}. It is the natural analog of the classical space of square-integrable functions~$L^2(\D)$ for the quantum disk.

The Hilbert space $H_{\tau_0}$ comes from the state $\tau_0(a) = \langle E_0, aE_0\rangle$. To describe it we f\/irst need to f\/ind all $a\in\mathcal{A}$ such that $\tau_0(a^*a) = 0$. A simple calculation yields
\begin{gather*}
\tau_0(a^*a) = \sum_{n\ge 0} |a_n^+(0)|^2.
\end{gather*}
Thus if $\tau_0(a^*a) = 0$ we get that $a_n^+(0) = 0$ for all $n\in\N$. Let $\mathcal{A}_{\tau_0} = \{a\in\mathcal{A} \colon \tau_0(a^*a) = 0\}$. Then we have
\begin{gather*}
\mathcal{A}/\mathcal{A}_{\tau_0} \cong \bigg\{a = \sum_{n\geq 0}U^n a^+_n(0)P_0\bigg\},
\end{gather*}
and $\|a\|_{\tau_0}^2 = \tau_0(a^*a)$. So, using the canonical basis $\{E_n := U^nP_0\}$ for $n\geq 0$, we can naturally identify $\mathcal{A}/\mathcal{A}_{\tau_0} $ with a dense subspace of $\ell^2(\N)$.

It is easy to describe the representation $\pi_{\tau_0} \colon A \to B(H_{\tau_0})$ of $A$ in the bounded operators on~$H_{\tau_0}$. We have
\begin{gather*}\pi_{\tau_0}(U)E_n = E_{n+1},\end{gather*}
 and
\begin{gather*}
\pi_{\tau_0}(a(\K))E_n = a(\K)U^nP_0 = U^na(\K+n)P_0 = U^na(n)P_0 =a(n)E_n.
\end{gather*}
Notice also that $\mathcal{A}/\mathcal{A}_{\tau_0} \ni [I] \mapsto P_0:=E_0$. In other words, $\pi_{\tau_0}$ is the def\/ining representation of the Toeplitz algebra~$A$.

Next we look at the GNS space associated with $\tau_\infty(a) = \lim\limits_{k\to\infty}\langle E_k , aE_k\rangle$. If $a(\K)\in\mathcal{A}$, we set
\begin{gather*}a_{\infty} = \lim_{k\to\infty} a(k).\end{gather*}
Again we want to f\/ind the subalgebra $\mathcal{A}_{\tau_\infty}$ of $a\in\mathcal{A}$ such that $\tau_\infty(a^*a) = 0$. A~direct computation shows that
\begin{gather*}
\tau_\infty(a^*a) = \sum_{n\ge 0} |a_{n,\infty}^+|^2 + \sum_{n\ge1}|a_{n,\infty}^-|^2,
\end{gather*}
so $\tau_\infty(a^*a) = 0$ if and only if $a_{n,\infty}^\pm = 0$ for all $n$. Now $\mathcal{A}/\mathcal{A}_{\tau_\infty} $ can be identif\/ied with a dense subspace of $L^2(S^1)$ by
\begin{gather*}
\mathcal{A}/\mathcal{A}_{\tau_\infty} \ni [a] = \bigg[a = \sum_{n\geq 0}U^n a^+_{n,\infty}+\sum_{n\geq 1}a^-_{n,\infty}(U^*)^n\bigg] \\
\hphantom{\mathcal{A}/\mathcal{A}_{\tau_\infty} \ni [a]}{} \mapsto\sum_{n\geq 0}a^+_{n,\infty}e^{inx}+\sum_{n\geq 1}a^-_{n,\infty}e^{-inx} := f_a(x).
\end{gather*}
Moreover we have
\begin{gather*}
\tau_\infty([a]) = \frac{1}{2\pi}\int_0^{2\pi}f_a(x) dx.
\end{gather*}

The representation $\pi_{\tau_\infty} \colon A \to B(H_{\tau_\infty})$ is easily seen to be given by
\begin{gather*}\pi_{\tau_\infty}(U)f(x) = e^{ix}f(x),\end{gather*}
and
\begin{gather*}
\pi_{\tau_\infty}(a(\K))f(x) = a_\infty f(x).
\end{gather*}
This completes the proof.
\end{proof}

\section{Implementations of derivations in quantum disk}\label{section6}
Let $H_{\tau}$ be the Hilbert space formed from the GNS construction on $A$ using a $\rho_\theta$-invariant state~$\tau$ and let $\pi_{\tau} \colon A \to B(H_{\tau })$ be the representation of $A$ in the bounded operators on $H_{\tau }$ via left multiplication, that is $\pi_{\tau }(a)f = [af]$. We have that $A\subset H_{\tau }$ is dense in $H_{\tau }$ and $[1]\in H_{\tau }$ is cyclic.

Let $\mathcal{D}_{\tau } = \pi_{\tau }(\mathcal{A})\cdot[1]$. Then $\mathcal{D}_{\tau }$ is dense in $H_{\tau }$. Def\/ine $U_{\tau ,\theta} \colon H_\tau\to H_\tau$ via $U_{\tau ,\theta}[a] = [\rho_\theta(a)]$. Notice for every $\theta$, the operator $U_{\tau ,\theta}$ extends to a unitary operator in $H_{\tau }$. Moreover by direct calculation we get
\begin{gather*}
U_{\tau ,\theta}\pi_\tau(a)U_{\tau ,\theta}^{-1} = \pi_{\tau }(\rho_\theta(a)).
\end{gather*}
It follows from the def\/initions that $U_{\tau ,\theta}(\mathcal{D}_{\tau })\subset\mathcal{D}_{\tau }$ and $\pi_\tau(\mathcal{A})(\mathcal{D}_{\tau })\subset\mathcal{D}_{\tau }$.

\subsection{Invariant derivations}
We f\/irst consider implementations of $\rho_\theta$-invariant derivations. Let $d_\beta$ be an invariant derivation $d_\beta\colon \mathcal{A}\to A$, $d_\beta(a) = [\beta(\K-1), a]$, as described in Proposition \ref{invar_der_rep}.

\begin{defin}
$D_{\tau} \colon \mathcal{D}_\tau \to H_\tau$ is called an {\it implementation} of a $\rho_\theta$-invariant derivation $d_\beta$ if $[D_{\tau}, \pi_\tau(a)] = \pi_\tau(d_\beta(a))$ and $U_{\tau,\theta}D_{\tau} U_{\tau,\theta}^{-1} = D_{\tau}$.
\end{defin}

In view of Theorem \ref{state_decomp} we implement the derivations on the three GNS Hilbert spaces~$H_{\tau_w}$, $H_{\tau_0}$ and~$H_{\tau_\infty}$.

\begin{prop} There exists a function $\alpha(k)$, $\sum\limits_{k\geq 0}|\beta(k-1)-\alpha(k)|^2w(k)<\infty$, such that any implementation $D_{\beta,\tau_w} \colon \mathcal{D}_{\tau_w} \to H_{\tau_w}$ of $d_\beta$ is uniquely represented by
\begin{gather}\label{InvImp}
D_{\beta,\tau_w}a = \beta(\K-1)a-a\alpha(\K).
\end{gather}
\end{prop}

\begin{proof}
We start by computing $U_{\tau_w,\theta}$. From the def\/initions we have
\begin{gather*}U_{\tau_w,\theta}(a) = \sum_{n\geq 0}U^n e^{in\theta}a^+_n(\K)+\sum_{n\geq 1}e^{-in\theta}a^-_n(\K)(U^*)^n.\end{gather*}

It follows from the assumptions that $D_{\beta,\tau_w}(I)$ must be invariant with respect to $U_{\tau_w,\theta}$. This implies that $D_{\beta,\tau_w}(I) = \eta(\K)$ for some diagonal operator $\eta(\K)\in H_{\tau_w}$. Thus, using Proposition~\ref{invar_der_rep}, we get
\begin{gather*}
D_{\beta,\tau_w}a = D_{\beta,\tau_w}\pi_{\tau_w}(a)\cdot I = [D_{\beta,\tau_w}, \pi_{\tau_w}(a)]\cdot I + \pi_{\tau_w}(a)D_{\beta,\tau_w}(1) = d_\beta(a) + a\eta(\K) \\
\hphantom{D_{\beta,\tau_w}a}{} = [\beta(\K-1),a] + a\eta(\K) = \beta(\K-1)a-a\alpha(\K),
\end{gather*}
where $\alpha(k)=\beta(k-1)-\eta(k)$. Notice also that $\eta(\K)\in H_{\tau_w}$ implies
\begin{gather*}||\eta(\K)||^2_{\tau_w}=\sum_{k\geq 0}|\beta(k-1)-\alpha(k)|^2w(k)<\infty.\end{gather*}
Conversely, it is easy to see that the operator def\/ined by~(\ref{InvImp}) is an implementation of $d_\beta$. Thus the result follows.
\end{proof}

\begin{prop}
There exists a number $c$ such that any implementation $D_{\beta,\tau_0} \colon \mathcal{D}_{\tau_0} \to \ell^2(\N )$ is of the form
\begin{gather*}
D_{\beta,\tau_0} = c\cdot I + \beta(\K-1),
\end{gather*}
where $\beta(k)$ is the convergent increment function from Proposition~{\rm \ref{invar_der_rep}}.
\end{prop}
\begin{proof}Again we need to f\/ind $U_{\tau_0,\theta}$. Since $\rho_\theta(U^nP_0)=e^{in\theta}U^nP_0$, we have
\begin{gather*}U_{\tau_0,\theta}E_n=e^{in\theta}E_n.\end{gather*}

Since $D_{\beta,\tau_0}E_0$ is invariant with respect to $U_{\tau_0,\theta}$, we must have $D_{\beta,\tau_0}E_0 = cE_0$ for some constant~$c$. Then
\begin{gather*}
D_{\beta,\tau_0}E_n = D_{\beta,\tau_0}U^nE_0 = (D_{\beta,\tau_0}U^n - U^nD_{\beta,\tau_0})E_0 + U^nD_{\beta,\tau_0}E_0 = d_\beta(U^n)E_0 + cU^nE_0.
\end{gather*}
By using Proposition~\ref{invar_der_rep} in the above equation we get
\begin{gather*}
D_{\beta,\tau_0}E_n = [\beta(\K-1), U^n]E_0 + cE_n = (\beta(\K-1)-\beta(\K-n-1))E_n + cE_n\\
\hphantom{D_{\beta,\tau_0}E_n}{} = (\beta(n-1) + c)E_n.
\end{gather*}
A short calculation verif\/ies that $D_{\beta,\tau_0}$ is indeed an implementation of~$d_\beta$. This completes the proof.
\end{proof}

\begin{prop}There exists a number $c$ such that the implementations $D_{\beta,\tau_\infty} \colon \mathcal{D}_{\tau_\infty} \to L^2(S^1)$ of $d_\beta$ are of the form
\begin{gather*}
D_{\beta,\tau_\infty} = \beta_\infty\frac{1}{i}\frac{d}{dx} + c,
\end{gather*}
where
\begin{gather*}\beta_\infty:=\lim_{k\to\infty}\left(\beta(k)-\beta(k-1)\right).\end{gather*}
\end{prop}
\begin{proof}
Like in the other proofs we need to understand what the value of $D_{\beta,\tau_\infty}$ on~$1$ is. A~simple calculation shows that
\begin{gather*}(U_{\tau_\infty,\theta} f)(x) = f(x - \theta).\end{gather*}
It is clear by the invariance properties that there exists a constant~$c$ such that $D_{\beta,\tau_\infty}(1) = c\cdot 1$.

Notice that $\mathcal{D}_{\tau_\infty}$ is the space of trigonometric polynomials on~$S^1$. By linearity we only need to look at $D_{\beta,\tau_\infty}$ on $e^{inx}$. We have
\begin{gather*}
D_{\beta,\tau_\infty}\big(e^{inx}\big) = \big(D_{\beta,\tau_\infty}\pi_{\tau_\infty}(U^n) - \pi_{\tau_\infty}(U^n)D_{\beta,\tau_\infty}\big)\cdot 1 + \pi_{\tau_\infty}(U^n)D_{\beta,\tau_\infty}(1) \\
\hphantom{D_{\beta,\tau_\infty}\big(e^{inx}\big)}{} = \big[D_{\beta,\tau_\infty}, \pi_{\tau_\infty}(U^n)\big] + \pi_{\tau_\infty}(U^n)D_{\beta,\tau_\infty}(1) = \pi_{\tau_\infty}(d_\beta(U^n)) + \pi_{\tau_\infty}(U^n)D_{\beta,\tau_\infty}(1) \\
\hphantom{D_{\beta,\tau_\infty}\big(e^{inx}\big)}{}
= \pi_{\tau_\infty}(U^n)\cdot\lim_{k\to\infty} (\beta(k+n) - \beta(k) ) + \pi_{\tau_\infty}(U^n)(1)c \\
\hphantom{D_{\beta,\tau_\infty}\big(e^{inx}\big)}{}
= e^{inx} (n\beta_\infty + c ) = \beta_\infty\frac{1}{i}\frac{d}{dx}\big(e^{inx}\big) + ce^{inx}.
\end{gather*}
It is again easy to verify that $D_{\beta,\tau_\infty}$ is an implementation. This completes the proof.
\end{proof}

\subsection{Covariant derivations}

Now let $\tilde d_\beta$ be a covariant derivation $\tilde d_\beta\colon \mathcal{A}\to A$ of the form $\tilde d_\beta(a) = [U\beta(\K), a]$, as proved in Proposition~\ref{covar_der_rep}. Let $\tau$ be a $\rho_\theta$-invariant state.

\begin{defin}
$\tilde D_{\tau} \colon \mathcal{D}_\tau \to H_\tau$ is called an {\it implementation} of a $\rho_\theta$-covariant derivation $\tilde d_\beta$ if $[\tilde D_{\tau}, \pi_\tau(a)] = \pi_\tau(\tilde d_\beta(a))$ and $U_{\tau,\theta}\tilde D_{\tau} U_{\tau,\theta}^{-1} = e^{i\theta}\tilde D_{\tau}$.
\end{defin}

We state without proofs the analogs of the above implementation results for covariant derivations; the verif\/ications are simple modif\/ications of the arguments for invariant derivations.

\begin{prop} There exists a function $\alpha(k)$, $\sum\limits_{k\geq 0}|\beta(k)-\alpha(k)|^2w(k)<\infty$, such that any implementation $\tilde D_{\beta,\tau_w} \colon \mathcal{D}_{\tau_w} \to H_{\tau_w}$ of $\tilde d_\beta$ is uniquely represented by
\begin{gather*}
\tilde D_{\beta,\tau_w}f = U\beta(\K)f - fU\alpha(\K).
\end{gather*}
\end{prop}

\begin{prop}The implementation $\tilde D_{\beta,\tau_0} \colon \mathcal{D}_{\tau_0} \to \ell^2(\N )$ of $\tilde d_\beta$ is of the form
\begin{gather*}
\tilde D_{\beta,\tau_0} = U\beta(\K),
\end{gather*}
i.e., on basis elements $\tilde D_{\beta,\tau_0}E_n = \beta(n)E_{n+1}$.
\end{prop}

\begin{prop}
There exists a number $c$ such that any implementation $\tilde D_{\beta,\tau_\infty} \colon \mathcal{D}_{\tau_\infty} \to L^2(S^1)$ of $\tilde d_\beta$ is of the form
\begin{gather*}
\tilde D_{\beta,\tau_\infty} = e^{ix}\left(\beta_\infty\frac{1}{i}\frac{d}{dx} + c\right),
\end{gather*}
where, as before,
$\beta_\infty:=\lim\limits_{k\to\infty} (\beta(k)-\beta(k-1) ).$
\end{prop}

\section{Compactness of parametrices}\label{section7}
\subsection{Spectral triples}

We say that a closed operator $D$ has compact parametrices if the operators $(I+D^*D)^{-1/2}$ and $(I+DD^*)^{-1/2}$ are compact. Other equivalent formulations are summarized in the appendix. Below we will reuse the same notation for the {\it closure} of the operators constructed in the previous section. In most cases it is very straightforward to establish when those operators have compact parametrices.

\begin{prop}The operators $D_{\beta,\tau_0}$, $\tilde D_{\beta,\tau_0}$ have compact parametrices if and only if \smash{$\beta(k)\!\to\!\infty$} as $k\to\infty$.
\end{prop}
\begin{proof} The operators $D_{\beta,\tau_0}$ are diagonal with eigenvalues $\beta(k-1)+c$, which must go to inf\/inity for the operators to have compact parametrices. The operators $\tilde D_{\beta,\tau_0}$ dif\/fer from the opera\-tors~$D_{\beta,\tau_0}$ by a shift, so they behave in the same way.
\end{proof}

\begin{prop}The operators $D_{\beta,\tau_\infty}$, $\tilde D_{\beta,\tau_\infty}$ have compact parametrices if and only if $\beta_\infty\ne 0$.
\end{prop}
\begin{proof} Similar to the proof of the proposition above, the operators $D_{\beta,\tau_\infty}$ are diagonal with eigenvalues $\beta_\infty n+c$, which go to inf\/inity if and only if $\beta_\infty\ne 0$.
\end{proof}

\begin{prop}The operators $D_{\beta,\tau_w}$ have compact parametrices if and only if
\begin{gather*}\beta(k+n-1)-\alpha(k)\to\infty \qquad \text{and} \qquad \beta(k-1)-\alpha(k+n)\to\infty
\end{gather*}
as $n,k \to\infty$.
\end{prop}
\begin{proof}
The operators $D_{\beta,\tau_w}$ can be diagonalized using the Fourier series
\begin{gather*}f = \sum_{n\geq 0}U^n f^+_n(\K)+\sum_{n\geq 1}f^-_n(\K)(U^*)^n.\end{gather*}
Computing $D_{\beta,\tau_w}f = \beta(\K-1)f-f\alpha(\K)$ we get\begin{gather*}
D_{\beta,\tau_w}f = \sum_{n\geq 0}U^n (\beta(\K+n-1)-\alpha(\K))f^+_n(\K)+\sum_{n\geq 1}(\beta(\K-1)-\alpha(\K+n))f^-_n(\K)(U^*)^n.
\end{gather*}
It follows that the numbers $\beta(k+n-1)-\alpha(k)$ and $\beta(k-1)-\alpha(k+n)$ are the eigenvalues of the diagonal operator, and must diverge for the operator to have compact parametrices.
\end{proof}

Let us remark that, in the last proposition, if for example $\alpha(k)=\beta(k-1)-i\eta(k)$, with $\beta_\infty$ and $\eta(k)$ real,
$\sum\limits_{k\geq 0}|\eta(k)|^2w(k)<\infty$, and $\eta(k)\to\infty$ as $k\to\infty$, then
\begin{gather*}\beta(k+n-1)-\alpha(k)\approx\beta_\infty n+i\eta(k)\to\infty,\end{gather*}
as $k, n\to \infty$. Similarly, we have
\begin{gather*}\beta(k-1)-\alpha(k+n)\approx-\beta_\infty k+i\eta(k+n)\to\infty,\end{gather*}
as $k, n\to \infty$.

\subsection{Covariant derivations and normal states}
Here we study the parametrices of the $\rho_\theta$-covariant operators which implement derivations in GNS Hilbert spaces $H_{\tau_w}$ corresponding to faithful normal states. In this section we enhance the notation for $\tilde D_{\beta,\tau_w}$; we will use instead\begin{gather*}
D_{\beta,\alpha,w}f = U\beta(\K)f - fU\alpha(\K),
\end{gather*}
a notation that clearly specif\/ies the coef\/f\/icients of the operator.
Denote by $D_{\beta,\alpha,w}^{\max}$ the closure of $D_{\beta,\alpha,w}$ def\/ined on $\mathcal{D}^{\max}_{\tau }= \pi_{\tau }(\mathcal{A})\cdot[1]$.

Def\/ine the $*$-algebra\begin{gather*}
\mathcal{A}_0 = \bigg\{a = \sum_{n\geq 0}U^n a^+_n(\K)+\sum_{n\geq 1}a^-_n(\K)(U^*)^n \colon a^\pm_n(k)\in {c_{00}},\ \text{f\/inite sums}\bigg\},
\end{gather*}
where $c_{00}$ are the sequences with compact support, i.e., eventually zero, and let $D_{\beta,\alpha,w}^{\min}$ be the closure of~$D_{\beta,\alpha,w}$ def\/ined on $\mathcal{D}^{\min}_{\tau } = \pi_{\tau }(\mathcal{A}_0)\cdot[1]$. Finally, will use the symbol $D_{\beta,\alpha,w}$ for any closed operator in $H_{\tau_w}$ such that $D_{\beta,\alpha,w}^{\min}\subset D_{\beta,\alpha,w}\subset D_{\beta,\alpha,w}^{\max}$.

The main objective of this section is to prove the following no-go result.

\begin{theo}There is no closed operator $D_{\beta,\alpha,w}$ in $H_{\tau_w}$, $D_{\beta,\alpha,w}^{\min}\subset D_{\beta,\alpha,w}\subset D_{\beta,\alpha,w}^{\max}$, with $\beta_\infty\ne 0$, such that $D_{\beta,\alpha,w}$ has compact parametrices.
\end{theo}

\begin{proof}
It is assumed below that $\beta_\infty\ne 0$. The outline of the proof is as follows. First, by a~sequence of equivalences, we show that the operator $D_{\beta,\alpha,w}$ has compact parametrices if and only if a simplif\/ied version of it has compact parametrices. Since in particular an operator with compact parametrices has to be Fredholm, the f\/initeness of the kernel and cokernel implies certain growth estimates on the parameters. Those estimates in turn let us compute parts of the spectrum of the Fourier coef\/f\/icients of $D_{\beta,\alpha,w}$ and that turns out to be not compatible with compactness of the parametrices.

First we show that $\beta(k)$ can be replaced by its absolute values. We will need the following information.

\begin{lem}
Let $\{\beta(k)\}$ be a sequence of complex numbers. If $\beta(k+1)-\beta(k)\to \beta_\infty$ and $\beta_\infty\neq0$, then there exists positive constants $c_1$ and $c_2$, and a nonnegative constant $c_3$ such that
\begin{gather*}
c_2(k+1)-c_3 \le |\beta(k)|\le c_1(k+1).
\end{gather*}
Moreover $\left||\beta(k+1)|-|\beta(k)|\right|$ is bounded.
\end{lem}

\begin{proof}
We will prove f\/irst that $\beta(k) = \beta_\infty\cdot(k+1)(1 + o(1))$.
From this the f\/irst inequality follows immediately. We decompose $\beta(k)$ as follows\begin{gather*}
\beta(k)=\beta_\infty\cdot(k+1)+\beta_0(k),
\end{gather*}
so that $\beta_0(k)-\beta_0(k-1)\to 0$ as $k\to\infty$, $\beta_0(-1)=0$. Using the notation $\psi(k):=\beta_0(k)-\beta_0(k-1)$, we want to show that\begin{gather*}
\frac{\beta_0(k)}{k+1}=\frac{1}{k+1}\sum_{j=0}^k\psi(j)\to 0
\end{gather*}
as $k\to\infty$. Given $\varepsilon>0$ f\/irst choose $j_\varepsilon$ so that $|\psi(j)|\leq\varepsilon$ for $j\geq j_\varepsilon$. First we split the sum{\samepage
\begin{gather*}\frac{1}{k+1}\sum_{j=0}^k|\psi(j)|=\frac{1}{k+1}\sum_{j=0}^{j_\varepsilon-1}|\psi(j)|+\frac{1}{k+1}\sum_{j=j_\varepsilon}^k|\psi(j)|\leq \frac{1}{k+1}\sum_{j=0}^{j_\varepsilon-1}|\psi(j)|+\varepsilon,
\end{gather*}
and then choose $k_\varepsilon$ so that $\frac{\sup |\psi(j)| j_\varepsilon}{k+1}\leq \varepsilon$ for $k\geq k_\varepsilon$. It follows that $\frac{\beta_0(k)}{k+1}=o(1)$.}

The second part of the lemma follows from the estimate
\begin{gather*} ||\beta(k+1)| - |\beta(k)| |\le |\beta(k+1) - \beta(k)| < \infty.\tag*{\qed}
\end{gather*}\renewcommand{\qed}{}
\end{proof}

\begin{lem}The operator $D_{\beta,\alpha,w}$ such that $D_{\beta,\alpha,w}^{\min}\subset D_{\beta,\alpha,w}\subset D_{\beta,\alpha,w}^{\max}$ has compact parametrices if and only if the operator $D_{|\beta|,\alpha,w}$ such that $D_{|\beta|,\alpha,w}^{\min}\subset D_{|\beta|,\alpha,w}\subset D_{|\beta|,\alpha,w}^{\max}$ has compact parametrices.
\end{lem}

\begin{proof}
Def\/ine the unitary operator $V(\K)$ by
\begin{gather*} V(k) = \exp \left(i\sum_{j=0}^{k-1}\operatorname{Arg}(\beta(j))\right),\end{gather*} and consider the following map $f\mapsto V(\K)f$ for $f\in\mathcal{H}_{\tau_w}$. This map preserves the domains $\mathcal{D}^{\min}_{\tau }$ and $\mathcal{D}^{\max}_{\tau }$. A~direct computation gives that
\begin{gather*}D_{|\beta|,\alpha,w} = V(\K)D_{\beta,\alpha,w}V(\K)^{-1}.\end{gather*}
This shows that $D_{\beta,\alpha,w}$ and $D_{|\beta|,\alpha,w}$ are unitarily equivalent, thus completing the proof.
\end{proof}

\begin{lem}\label{para_shift} The operator $D_{\beta,\alpha,w}$ such that $D_{\beta,\alpha,w}^{\min}\subset D_{\beta,\alpha,w}\subset D_{\beta,\alpha,w}^{\max}$ has compact parametrices if and only if the operator $D_{\beta+\gamma_1,\alpha+\gamma_2,w}$ such that $D_{\beta+\gamma_1,\alpha+\gamma_2,w}^{\min}\subset D_{\beta+\gamma_1,\alpha+\gamma_2,w}\subset D_{\beta+\gamma_1,\alpha+\gamma_2,w}^{\max}$ has compact parametrices for any constants $\gamma_1$ and $\gamma_2$.
\end{lem}

\begin{proof}Notice that the dif\/ference $D_{\beta+\gamma_1,\alpha+\gamma_2,w}-D_{\beta,\alpha,w}$ is bounded, hence the two operators both either have or do not have compact parametrices simultaneously, see Appendix~\ref{appendixA}.
\end{proof}

It follows from those lemmas that, without loss of generality, we may assume that $\beta(k)>0$, where $\beta(k)$ satisf\/ies inequalities
\begin{gather}
c_2(k+1) \le \beta(k)\le c_1(k+1),\nonumber\\
|\beta(k+1) - \beta(k)| < \infty,\label{growth_cond}
\end{gather}
$c_1$ and $c_2$ positive.

Next we look at properties of $\alpha$. For a f\/inite sum\begin{gather*}
f = \sum_{n\geq 0}U^n f^+_n(\K)+\sum_{n\geq 1}f^-_n(\K)(U^*)^n
\end{gather*}
in the domain of the operator $D_{\beta,\alpha,w}$ we can write $D_{\beta,\alpha,w}f = U\beta(\K)f - fU\alpha(\K)$ in Fourier components as\begin{gather*}
D_{\beta,\alpha,w}f = \sum_{n\geq 0}U^{n+1} (D_n^+f^+_n)(\K)+\sum_{n\geq 1}(D_n^-f^-_n)(\K)(U^*)^{n-1},
\end{gather*}
where\begin{gather*}
(D_n^+f)(k) = \beta(k+n)f(k) - \alpha(k)f(k+1), \\
(D_n^-f)(k) = \alpha(k+n-1)f(k) - \beta(k-1)f(k-1). \end{gather*}

\begin{lem}If the operator $D_{\beta,\alpha,w}$ such that $D_{\beta,\alpha,w}^{\min}\subset D_{\beta,\alpha,w}\subset D_{\beta,\alpha,w}^{\max}$ has compact parametrices, then $\dim \operatorname{coker} (D_{\beta,\alpha,w}^{\max})<\infty$ and $\alpha(k)$ has at most f\/initely many zeros.
\end{lem}

\begin{proof}First note that since $D_{\beta,\alpha,w}$ has compact parametrices it is a Fredholm operator, so it has f\/inite-dimensional cokernel. This means that $D_{\beta,\alpha,w}^{\max}$ has f\/inite-dimensional cokernel since $\ker (D_{\beta,\alpha,w}^{\max})^*\subset\ker (D_{\beta,\alpha,w})^*$. Next suppose that $\alpha(k)$ has inf\/initely many zeros and then try to compute $\ker (D_{\beta,\alpha,w}^{\max})^*$. In Fourier components this leads to the following equations
\begin{gather*}
(D_n^-)^*f(k) = \overline{\alpha}(k+n-1)f(k) - {\beta}(k)f(k+1) = 0.
\end{gather*}
Suppose $\overline{\alpha}(N)=0$ for some $N\ge0$, and consider $n=N+1$. Solving recursively the equation $(D_{N+1}^-)^*f(k)=0$ gives that $f(k)=0$ for all $k\ge1$. Thus the function
\begin{gather}\label{chi_zero_def}
\chi_0(k) =
\begin{cases}
1, & \text{for }k=0,\\
0, &\text{for }k\ge1
\end{cases}
\end{gather}
belongs to the kernel of $(D_{N+1}^-)^*$, and, because of the f\/inite support, $\chi_0(k)$ is in the domain of~$((D_{N+1}^-)^*)^{\min}$. This works for any $N\ge0$ such that ${\alpha}(N)=0$, producing an inf\/inite-dimensional cokernel for~$D_{\beta,\alpha,w}^{\max}$, contradicting the assumption. Thus the result follows.
\end{proof}

As a consequence of the above lemma and also Lemma~\ref{para_shift} we will assume from now on that $\alpha(k)\ne 0$ for every~$k$.

We f\/ind it convenient to work with unweighted Hilbert spaces. This is achieved by means of the following lemma.

\begin{lem}Let $\mathcal{H}_{\tau_w}$ be the weighted Hilbert space of Proposition~{\rm \ref{GNS_prop}(1)}, and let $\mathcal{H}$ be that Hilbert space for which the weight $w(k)=1$. The operator $D_{\beta,\alpha,w}$ such that $D_{\beta,\alpha,w}^{\min}\subset D_{\beta,\alpha,w}\subset D_{\beta,\alpha,w}^{\max}$ has compact parametrices if and only if the operator $D_{\beta,\tilde\alpha,1}$ such that $D_{\beta,\tilde\alpha,1}^{\min}\subset D_{\beta,\tilde\alpha,1}\subset D_{\beta,\tilde\alpha,1}^{\max}$ has compact parametrices,
where\begin{gather*}
\tilde{\alpha}(k) = \alpha(k)\frac{\sqrt{w(k)}}{\sqrt{w(k+1)}}.
\end{gather*}
\end{lem}

\begin{proof}In $\mathcal{H}_{\tau_w}$ write the norm as
\begin{gather*}
\|f\|_w^2 = \operatorname{tr}(w(\K)f^*f) = \operatorname{tr} \big(fw(\K)^{1/2}\big)^*\big(w(\K)^{1/2}f\big),
\end{gather*}
and set $\varphi(f) = fw(\K)^{1/2}\colon \mathcal{H}_{\tau_w}\to\mathcal{H}$. Then $\varphi$ is a bounded operator with bounded inverse, and is in fact an isomorphism of Hilbert spaces. Moreover, we have
\begin{gather*}
\varphi\big(D_{\beta,\alpha,w}\varphi^{-1}f\big) = U\beta(\K)f - fU\alpha(\K)\frac{\sqrt{w(\K)}}{\sqrt{w(\K+1)}}=D_{\beta,\tilde\alpha,1}f,
\end{gather*}
and $\varphi D_{\beta,\alpha,w}\varphi^{-1}\colon \mathcal{H}\to\mathcal{H}$. So $D_{\beta,\alpha,w}$ and $D_{\beta,\tilde\alpha,1}$ are unitarily equivalent, thus completing the proof. Notice also that $\tilde{\alpha}(k)\ne 0$, because ${\alpha}(k)\ne 0$.
\end{proof}

From now on we will work with operators $D_{\beta,\alpha,1}^{\min}\subset D_{\beta,\alpha,1}\subset D_{\beta,\alpha,1}^{\max}$ in the unweighted Hilbert space~$\mathcal{H}$. For convenience we def\/ine a sequence $\{\mu(k)\}$ such that $\mu(0) = 1$ and
\begin{gather*}
\alpha(k) = \beta(k)\frac{\mu(k+1)}{\mu(k)}.
\end{gather*}
Such $\mu(k)$ is completely determined by the above equation in terms of $\alpha$ and $\beta$ and will be used as a coef\/f\/icient instead of $\alpha$. We rewrite the four main operators as follows
\begin{gather*}
(D_n^+f)(k) = \beta(k+n)\left(f(k) - \frac{\beta(k)}{\beta(k+n)}\frac{\mu(k+1)}{\mu(k)}f(k+1)\right), \\
(D_n^-f)(k) = \beta(k+n-1)\frac{\mu(k+n)}{\mu(k+n-1)}\left(f(k) - \frac{\beta(k-1)}{\beta(k+n-1)}\frac{\mu(k+n-1)}{\mu(k+n)}f(k-1)\right), \\
((D_n^+)^*)f(k) = \overline{\beta}(k+n)\left(f(k) - \frac{\overline{\beta}(k)}{\overline{\beta}(k+n)}\frac{\overline{\mu}(k)}{\overline{\mu}(k-1)}f(k-1)\right), \\
((D_n^-)^*f)(k) = \overline{\beta}(k+n-1)\frac{\overline{\mu}(k+n)}{\overline{\mu}(k+n-1)}\left(f(k) - \frac{\overline{\beta}(k)}{\overline{\beta}(k+n-1)}\frac{\overline{\mu}(k+n-1)}{\overline{\mu}(k+n)}f(k+1)\right).
\end{gather*}

Next, using Fourier components above, we study the kernel and the cokernel of $D_{\beta,\alpha,1}$.

\begin{lem}\label{formal_kernels}
The formal kernels of $D_n^+$ and $(D_n^-)^*$ are one-dimensional and are spanned by, correspondingly
\begin{gather*}
f_n^+(k) =\frac{1}{\mu(k)}\prod_{j=0}^{n-1}\beta(j+k),
\end{gather*}
and
\begin{gather*}
f_n^-(k) = \overline{\mu}(k+n-1)\prod_{j=0}^{n-1}\overline{\beta}(j+k).
\end{gather*}
The operators $D_n^-$ and $(D_n^+)^*$ have no algebraic kernel; consequently they have no kernel at all.
\end{lem}

\begin{proof}
We f\/irst study $D_n^+f(k)=0$; the calculations are the same as in \cite{KM1}. Solving the equation $D_n^+f(k)=0$ recursively, we arrive at
\begin{gather*}
f(k) = \prod_{j=0}^{k-1}\frac{\beta(j+n)}{\beta(j)}\frac{\mu(0)}{\mu(k)}f(0) = \frac{f(0)}{\mu(k)}\frac{\beta(n)\cdots\beta(k+n-1)}{\beta(0)\cdots\beta(k-1)} \\
\hphantom{f(k)}{} =\frac{f(0)}{\mu(k)}\frac{\beta(n)\cdots\beta(k+n-1)\cdot\beta(0)\cdots\beta(n-1)}{\beta(0)\cdots\beta(k-1)\cdot\beta(0)\cdots\beta(n-1)} = \frac{f(0)}{\mu(k)}\frac{\beta(k)\cdots\beta(k+n-1)}{\beta(0)\cdots\beta(n-1)}.
\end{gather*}
Other calculations are similar. This completes the proof.
\end{proof}

The computations above were formal; to actually compute the kernel and the cokernel of~$D_{\beta,\alpha,1}$ we need to look at only those solutions which are in the domain/codomain of $D_{\beta,\alpha,1}$. It is important to keep in mind the following inclusions\begin{gather*}
\ker D_{\beta,\alpha,1}^{\min}\subset \ker D_{\beta,\alpha,1}\subset \ker D_{\beta,\alpha,1}^{\max},
\end{gather*}
and \begin{gather*}
\operatorname{coker}D_{\beta,\alpha,1}^{\max}\subset \operatorname{coker}D_{\beta,\alpha,1}\subset \operatorname{coker}D_{\beta,\alpha,1}^{\min}.
\end{gather*}

The following lemma exhibits the f\/irst key departure from the analogous classical analysis of the d-bar operator.

\begin{lem}\label{kernel_property}
If the operator $D_{\beta,\alpha,1}$ such that $D_{\beta,\alpha,1}^{\min}\subset D_{\beta,\alpha,1}\subset D_{\beta,\alpha,1}^{\max}$ has compact parametrices, then both $\ker D_{\beta,\alpha,1}^{\max}$ and $\operatorname{coker}D_{\beta,\alpha,1}^{\min}$ are finite-dimensional.

Moreover, the sums
\begin{gather*}
\sum_{k=0}^\infty\left|\prod_{j=0}^{n-1}\beta(j+k)\right|^2|\mu(k+n-1)|^2 \qquad\text{and}\qquad\sum_{k=0}^\infty\left|\prod_{j=0}^{n-1}\beta(j+k)\right|^2\frac{1}{|\mu(k)|^2}
\end{gather*}
are both infinite for all $n\geq n_0$.
\end{lem}

\begin{proof}Let $f_n^+$ and $f_n^-$ be solutions to the equations $D_n^+f = 0$ and $(D_n^-)^*f = 0$ respectively, as described in Lemma~\ref{formal_kernels}. First we study $D_n^+f=0$. There are two options
\begin{enumerate}\itemsep=0pt
\item[$(1)$] $\|f_n^+\|<\infty$ for all $n$, or
\item[$(2)$] there exists $n_0\ge0$ such that $\|f_{n_0}^+\|=\infty$.
\end{enumerate}

Consider the f\/irst option f\/irst, i.e.,
\begin{gather*}
\sum_{k=0}^\infty |f_n^+(k)|^2 = \sum_{k=0}^\infty \frac{\beta(k)^2\cdots\beta(k+n-1)^2}{|\mu(k)|^2}<\infty
\end{gather*}
for every $n$, which implies that $D_{\beta,\alpha,1}^{\max}$ has an inf\/inite-dimensional kernel. We argue below that in this case the kernel of $D_{\beta,\alpha,1}^{\min}$ is also inf\/inite-dimensional, which is not true in classical theory. Consider the sequence
\begin{gather*}
f_N(k) =
\begin{cases}
f_n^+(k), &\text{for }k\le N, \\
0, &\text{else}.
\end{cases}
\end{gather*}
Notice that, because it is eventually zero, the sequence $f_N(k)$ is in the domain of $(D_n^+)^{\min}$ and $f_N\to f_n^+$ in $\ell^2(\N)$ as $N\to\infty$. Moreover, a direct calculation shows that
\begin{gather*}
D_n^+f_N(k) = \begin{cases}
 \beta(n+N)f_n^+(N)= f_{n+1}^+(N), &\text{for }k=N, \\
0, &\text{else}.
\end{cases}
\end{gather*}
From this we see that $D_n^+f_N\to0$ as $N\to\infty$ since $\|f_n^+\|<\infty$ for all $n$. This shows that the formal kernel of~$(D_n^+)$ is contained in the domain of $(D_n^+)^{{min}}$. This implies that $D_{\beta,\alpha,1}$ has an inf\/inite-dimensional kernel contradicting the fact that $D_{\beta,\alpha,1}$ is Fredholm. A similar argument produces an inf\/inite-dimensional cokernel for $D_{\beta,\alpha,1}$ by studying option~$(1)$ for $(D_n^-)^*f = 0$. Consequently, option $(1)$ does not happen in our case, and option~$(2)$ must be true. It is clear from the growth conditions \eqref{growth_cond} that if there exists $n_0$ such that $\|f_{n_0}^{\pm}\|=\infty$ then $\|f^{\pm}_n\|=\infty$ for all $n\ge n_0$. But that means that the $\ell^2(\N)$ kernels of $(D_n^{\pm})$ and $(D_n^{\pm})^*$ are all zero for $n$ large enough.
This implies that both $\ker D_{\beta,\alpha,1}^{\max}$ and $\operatorname{coker}D_{\beta,\alpha,1}^{\min}$ are f\/inite-dimensional.
Moreover, $\|f^{\pm}_n\|=\infty$ for all $n\ge n_0$ gives the divergence of the sums in the statement of the lemma. Thus the proof is complete.
\end{proof}

It follows from the above lemma, and from the remarks right before it, that all three operators $D_{\beta,\alpha,1}^{\min}\subset D_{\beta,\alpha,1}\subset D_{\beta,\alpha,1}^{\max}$ have compact parametrices.

Next we discuss the inverses of $D_n^\pm$ and their formal adjoints. Operators $D_n^-$ and $(D_n^+)^*$ have no formal kernels and can be inverted on any domain of sequences. The other operators preserve $c_0\subset \ell^2(\N)$ and can be inverted on~$c_0$. The corresponding formulas are
\begin{gather*}
(D_n^+)^{-1}g(k) = \sum_{j=k}^\infty\frac{\beta(k)\cdots\beta(k+n-1)}{\beta(j)\cdots\beta(j+n)}\cdot\frac{\mu(j)}{\mu(k)}g(j),\\
(D_n^-)^{-1}g(k) = \begin{cases}
\displaystyle \sum_{j=0}^k \frac{\beta(j)\cdots\beta(j+n-2)}{\beta(k)\cdots\beta(k+n-1)}\cdot\frac{\mu(j+n-1)}{\mu(k+n)}g(j) &\text{if }n\ge2,\\
\displaystyle \frac{1}{\beta(k)\mu(k+1)}\sum_{j=0}^k\mu(j)g(j) &\text{if }n=1,
\end{cases}\\
((D_n^+)^*)^{-1}g(k) = \sum_{j=0}^k \frac{\overline{\beta}(j)\cdots\overline{\beta}(j+n-1)}{\overline{\beta}(k)\cdots\overline{\beta}(k+n)}
\cdot\frac{\overline{\mu}(k)}{\overline{\mu}(j)}g(j),
\\
((D_n^-)^*)^{-1}g(k) =
\begin{cases}
\displaystyle \sum_{j=k}^\infty \frac{\overline{\beta}(k)\cdots\overline{\beta}(k+n-2)}{\overline{\beta}(j)\cdots\overline{\beta}(j+n-1)}\cdot \frac{\overline{\mu}(k+n-1)}{\overline{\mu}(j+n-1)}g(j) &\text{if }n\ge2,\\
\displaystyle \overline{\beta}(k)\overline{\mu}(k)\sum_{j=k}^\infty \frac{1}{\overline{\mu}(j)}g(j) &\text{if }n=1.
\end{cases}
\end{gather*}

Using those formulas we obtain key growth estimates on coef\/f\/icients $\mu(k)$ in the following lemma.
\begin{lem}
If the operator $D_{\beta,\alpha,1}$ such that $D_{\beta,\alpha,1}^{\min}\subset D_{\beta,\alpha,1}\subset D_{\beta,\alpha,1}^{\max}$ has compact parametrices then, for $n$ large enough, $(D_n^\pm)^{\min}$ and $((D_n^\pm)^*)^{\max}$ are invertible operators with bounded inverses. Moreover, there exists a number $n_1\ge0$ and a constant $C$ such that
\begin{gather}\label{mu_ineq}
\frac{1}{C(k+1)^{n_1}}\le |\mu(k)|\le C(k+1)^{n_1}
\end{gather}
for $n\ge n_1$.
\end{lem}

\begin{proof}The Fredholm property of $D_{\beta,\alpha,1}$ implies that the ranges of $D_{\beta,\alpha,1}^{\max}$ and $((D_{\beta,\alpha,1})^*)^{\max}$ are closed. By the proof of Lemma~\ref{kernel_property} there exists $n_0\ge0$ such that for all $n\ge n_0$ the $\ell^2(\N)$ kernels of $D_n^\pm$ and $(D_n^\pm)^*$ are zero. It follows that $\operatorname{Ran}(D_n^-)^{\max} = \operatorname{Ran}((D_n^+)^*)^{\max} = \ell^2(\N)$ for $n\ge n_0$. In particular this says $((D_n^-)^{\max})^{-1}\chi_0(k)\in\ell^2(\N)$, where $\chi_0(k)$ was def\/ined in \eqref{chi_zero_def}. As a consequence we obtain
\begin{gather*}
\sum_{k=0}^\infty \frac{1}{\beta(k)^2\cdots\beta(k+n-1)^2|\mu(k+n)|^2}<\infty.
\end{gather*}
Using the growth conditions \eqref{growth_cond} the inequality above yields
\begin{gather*}
\sum_{k=0}^\infty \frac{1}{(k+1)^2\cdots (k+n)^2|\mu(k+n)|^2}< \infty
\end{gather*}
for $n\geq n_0$, which gives the left hand side of the inequality~\eqref{mu_ineq}. To the get the right-hand side, we use
$(((D_n^+)^{\max})^*)^{-1}\chi_0(k)\in\ell^2(\N)$.
\end{proof}

Now that we have control over the coef\/f\/icients of $D_{\beta,\alpha,1}$ we can compute the spectrum of its Fourier coef\/f\/icients. Notice that, using Proposition~\ref{sandwich} in the appendix, we have that since~$D_{\beta,\alpha,1}^{\min}$ and~$D_{\beta,\alpha,1}^{\max}$ are Fredholm, and $D_{\beta,\alpha,1}$ has compact parametrices, then both~$D_{\beta,\alpha,1}^{\min}$ and~$D_{\beta,\alpha,1}^{\max}$ also have compact parametrices. The following calculations are similar to the calculations in~\cite{BHS} for the Cesaro operator.

\begin{lem}\label{spec_lemma}
The continuous spectrum $\sigma_c$, the point spectrum $\sigma_p$, and the residual spect\-rum~$\sigma_r$, of the operator $(D_n^+)^{\max}$ have the following properties:
\begin{enumerate}\itemsep=0pt
\item[$1)$] $\sigma_c((D_n^+)^{\max}) = \varnothing$,
\item[$2)$] $\{\lambda \in\C \colon \operatorname{Re}\lambda \gg 0\} \subset \sigma_p((D_n^+)^{\max})$,
\item[$3)$] $\{\lambda\in\C \colon \operatorname{Re}\lambda \ll 0\} \not\subset\sigma_p((D_n^+)^{\max})$,
\item[$4)$] $\sigma_r((D_n^+)^{\max}) = \varnothing$ or has at most finitely many spectral values.
\end{enumerate}
\end{lem}

\begin{proof}
The Fredholm property of $D_{\beta,\alpha,1}^{\max}$ implies that $\operatorname{Ran}((D_n^+)^{\max} - \lambda I)$ is closed, meaning that $\sigma_c((D_n^+)^{\max}) = \varnothing$.

Next we study the eigenvalue equation $(D_n^+f)(k) = \lambda f(k)$, that is
\begin{gather*}
\beta(k+n)f(k) - \beta(k)\frac{\mu(k+1)}{\mu(k)}f(k+1) = \lambda f(k).
\end{gather*}
This equation can be easily solved, yielding a one-parameter solution generated by
\begin{gather*}
f_\lambda(k) = \prod_{j=0}^{k-1}\left(\frac{\beta(j+n)-\lambda}{\beta(j)}\right)\frac{1}{\mu(k)} = \prod_{j=0}^{k-1}\left(1 + \frac{\beta(j+n)-\beta(j) -\lambda}{\beta(j)}\right)\frac{1}{\mu(k)}.
\end{gather*}
The question then is when does $f_\lambda\in\ell^2(\N)$? To study estimates on $f_\lambda(k)$ we use the following three simple inequalities
\begin{gather}\label{classic_inequal}
\begin{aligned}
&1)\quad 1 + x\le e^x, \\
&2)\quad \ln(x)=\int_1^k\frac{1}{x} dx \le \sum_{j=0}^k\frac{1}{j+1}\le 1 + \int_1^k\frac{1}{x} dx=1+\ln(x), \\
&3)\quad \text{there exists a constant }C_\varepsilon\ \text{such that }C_\varepsilon e^{(1-\varepsilon)x}\le 1+x,\\
&\quad\quad\text{for }0<\varepsilon< 1 \ \text{and small } |x|.
\end{aligned}
\end{gather}
First we estimate from above each factor in the formula for $f_\lambda$ as follows
\begin{gather*}
\left|1 + \frac{\beta(j+n)-\beta(j)-\lambda}{\beta(j)}\right|^2 =1 + 2\operatorname{Re}\left(\frac{\beta(j+n)-\beta(j)-\lambda}{\beta(j)}\right) +\left|\frac{\beta(j+n)-\beta(j)-\lambda}{\beta(j)}\right|^2 \\
\qquad{}
\le \exp \left(2\operatorname{Re}\left(\frac{\beta(j+n)-\beta(j)-\lambda}{\beta(j)}\right) +\left|\frac{\beta(j+n)-\beta(j)-\lambda}{\beta(j)}\right|^2\right),
\end{gather*}
where we used inequality $1)$ of equation~(\ref{classic_inequal}). This implies that
\begin{gather*}
|f_\lambda(k)|\le \exp \left[\sum_{j=0}^{k-1}\operatorname{Re}\left(\frac{\beta(j+n)-\beta(j)-\lambda}{\beta(j)}\right)+ \frac{1}{2}\sum_{j=0}^{k-1}\left|\frac{\beta(j+n)-\beta(j)-\lambda}{\beta(j)}\right|^2\right]\frac{1}{|\mu(k)|}.
\end{gather*}
Notice that for f\/ixed $\lambda$ we have $|\beta(j+n)-\beta(j) -\lambda|\le \operatorname{const}$ by~\eqref{growth_cond} and, because $c_2(j+1)\le \beta(j)\le c_1(j+1)$, we obtain
\begin{gather*}
\sum_{j=0}^{k-1} \left|\frac{\beta(j+n)-\beta(j)-\lambda}{\beta(j)}\right|^2\le \sum_{j=0}^{k-1}\frac{\operatorname{const}}{(j+1)^2}< \operatorname{const},
\end{gather*}
which accounts for the second term in the exponent. To estimate the f\/irst term we also use~$2)$ of equation~(\ref{classic_inequal}) to get\begin{gather*}
|f_\lambda(k)| \le \exp \left(\sum_{j=0}^{k-1}\frac{\operatorname{const} - \operatorname{Re}\lambda}{\operatorname{const} (j+1)}\right)\frac{1}{|\mu(k)|}\le \operatorname{const} (k+1)^{(\operatorname{const} - \operatorname{Re}\lambda)},
\end{gather*}
where we applied~\eqref{mu_ineq} to estimate~$\mu(k)$. This last inequality implies that $f_\lambda(k)\in\ell^2(\N)$ if $\operatorname{Re}\lambda \gg 0$. This shows that
\begin{gather*}\{\lambda \in\C \colon \operatorname{Re}\lambda \gg 0\} \subset \sigma_p\big((D_n^+)^{\max}\big).\end{gather*}

Next we estimate $f_\lambda$ from below by using part~$3)$ of equation~(\ref{classic_inequal}) with \begin{gather*}
x=2 \frac{(\beta(j+n)-\beta(j)-\operatorname{Re}\lambda)}{\beta(j)},
\end{gather*}
which, by previous discussion, is small for~$j$ large enough. We get the following estimate
\begin{gather*}
\left|1+ \frac{\beta(j+n)-\beta(j)-\lambda}{\beta(j)}\right|^2 \ge 1 + 2\ \frac{\beta(j+n)-\beta(j)-\operatorname{Re}\lambda}{\beta(j)} \\
\hphantom{\left|1+ \frac{\beta(j+n)-\beta(j)-\lambda}{\beta(j)}\right|^2}{}
\ge \exp \left(2(1-\varepsilon)\frac{(\beta(j+n)-\beta(j)-\operatorname{Re}\lambda)}{\beta(j)}\right),
\end{gather*}
valid for large $j$. By using the conditions on $\beta(j)$ and $\mu(k)$, and also $2)$ of equation~(\ref{classic_inequal}), we get
\begin{gather*}
|f_\lambda(k)| \ge \frac{\operatorname{const}}{|\mu(k)|}k^{(1-\varepsilon)(\operatorname{const}-\operatorname{Re}\lambda)} \ge (\operatorname{const})k^{(1-\varepsilon)(\operatorname{const}-\operatorname{Re}\lambda)}.
\end{gather*}
This inequality shows that if $\operatorname{Re}\lambda\ll 0$ then $f(k)\notin\ell^2(\N)$. This in turn implies that
\begin{gather*}\{\lambda\in\C \colon \operatorname{Re}\lambda \ll 0\} \not\subset\sigma_p\big((D_n^+)^{\max}\big).\end{gather*}

Finally, to determine the residual spectrum of $(D_n^+)^{\max}$, we consider the eigenvalue equation $(D_n^+)^*f(k) = \lambda f(k)$, which is the same as
\begin{gather*}
\beta(k+n)f(k) - \beta(k-1)\frac{\overline{\mu}(k)}{\overline{\mu}(k-1)}f(k-1) = \lambda f(k).
\end{gather*}
Rearranging the terms in the above equation yields
\begin{gather*}
[\beta(k+n)-\lambda]\frac{f(k)}{\overline{\mu}(k)} = \beta(k-1)\frac{f(k-1)}{\overline{\mu}(k-1)}.
\end{gather*}
This equation has non-trivial solutions if and only if $\beta(k+n)-\lambda = 0$ for some $k$, which can only happen for specif\/ic values of $\lambda$. Namely, if $\lambda_l = \beta(l+n)$, then the above equation recursively gives $f(0)=f(1) = \cdots = f(l-1) = 0$ and
\begin{gather*}f(k+l) = \operatorname{const} \beta(k)\cdots\beta(k+l-1) \overline{\mu}(l+k).\end{gather*}
If $l$ is large enough then $f(k)\notin\ell^2(\N)$. This means that the residual spectrum of $D_n^+$ has at most f\/initely many values or is empty, proving the remaining part of the lemma, thus completing the proof of the lemma.
\end{proof}

We can now easily f\/inish the proof of the theorem. As explained in appendix, if $D_{\beta,\alpha,1}^{\max}$ has compact parametrices then its spectrum is either empty, the whole plane $\C$, or consists of eigenvalues going to inf\/inity. Clearly this is not consistent with Lemma~\ref{spec_lemma}, and hence $D_{\beta,\alpha,1}$ does not have compact parametrices.
\end{proof}

\appendix
\section{Appendix}\label{appendixA}
The main objective of this appendix is to review some generalities about unbounded operators with compact parametrices. Presumably all of the statements below are known, however they don't seem to appear together in any one reference.

Throughout this appendix $D$ is a closed unbounded operator in a separable Hilbert space. Recall that~$D$ is called a Fredholm operator if there are bounded operators $Q_1$ and~$Q_2$ such that $Q_1D-I$ and $DQ_2-I$ are compact. The operators~$Q_1$ and~$Q_2$ are called left and right parametrices respectively. Equivalently, $D$ is a Fredholm operator if the kernel and the cokernel of $D$ are f\/inite-dimensional. A~Fredholm operator always has a single parametrix, i.e., a bounded operator $Q$ such that $QD-I$ and $DQ-I$ are compact. In the literature the case of unbounded Fredholm operators is usually not discussed directly, however a~closed operator can be considered as a~bounded operator on its domain equipped with the graph inner product $||x||_D^2=||x||^2+||Dx||^2$. A~good reference on Fredholm operators is~\cite{Sc}.

We say that a closed, Fredholm operator $D$ has {\it compact parametrices} if at least one of the parametrices $Q_1$ and $Q_2$ is compact. By applying $Q_1$ to $DQ_2-I$ on the left and $Q_2$ to $Q_1D-I$ on the right, we see that if one of the parametrices $Q_1$ and $Q_2$ is compact so is the other. Similarly, if $Q'_1$ and $Q'_2$ is another set of parametrices of an operator with compact parametrices, then both $Q'_1$ and $Q'_2$ must be comopact.
It is sometimes easier to construct separate left and right parametrices rather then a two-sided parametrix.

Our f\/irst task is to work out several equivalent def\/initions of an operator with compact parametrices.

For $\lambda$ in the resolvent set $\rho(D)$ let $R_D(\lambda)=(D-\lambda I)^{-1}$ be the resolvent operator.

\begin{prop}
Suppose $\rho(D)\neq\varnothing$ and $R_D(\lambda)$ is compact for some $\lambda\in\rho(D)$, then $R_D(\mu)$ is compact for every $\mu\in\rho(D)$.
\end{prop}
\begin{proof} This immediately follows from the resolvent identity.
\end{proof}
First we rephrase the concept of an operator with compact parametrices in terms of resolvents.

\begin{prop}
Suppose $\rho(D)\neq\varnothing$ and $R_D(\lambda)$ is compact for some $\lambda\in\rho(D)$, then $D$ is a Fredholm operator with compact parametrices. Conversly, if $D$ is a Fredholm operator with compact parametrices and $\rho(D)\neq\varnothing$ then $R_D(\lambda)$ is compact.
\end{prop}

\begin{proof}Consider the following calculations
\begin{gather*}
DR_D(\lambda) = (D-\lambda I)R_D(\lambda) + \lambda R_D(\lambda) = I + \lambda R_D(\lambda),
\end{gather*}
and
\begin{gather*}
R_D(\lambda)D = R_D(\lambda)(D-\lambda I) + R_D(\lambda)\lambda = I + \lambda R_D(\lambda).
\end{gather*}
So, if $R_D(\lambda)$ is compact for $\lambda\in\rho(D)$ then $D$ is Fredholm with parametrix $\lambda R_D(\lambda)$ which is compact. Conversely, if $D$ is a~Fredholm operator with compact parametrices and $\rho(D)\neq\varnothing$ then $R_D(\lambda)$ is compact as a parametrix of~$D$.
This completes the proof.
\end{proof}

Next we give a characterization of operators with compact parametrices in terms of self-adjoint operators $D^*D$ and $DD^*$.
\begin{prop}\label{ddstarprop}
Suppose $(I+D^*D)^{-1/2}$ and $(I+DD^*)^{-1/2}$ are both compact. Then $D$ is a Fredholm operator with compact parametrices. Conversely, if $D$ is a Fredholm operator with compact parametrices then $(I+D^*D)^{-1/2}$ and $(I+DD^*)^{-1/2}$ are both compact operators.
\end{prop}

\begin{proof}
We construct the parametrices of $D$ explicitly. To this end consider the operator
\begin{gather*}Q := D^*(I+DD^*)^{-1}.\end{gather*}
Notice that, since $(I+DD^*)^{-1/2}$ is compact, $(I+DD^*)^{-1}$ is compact. Moreover, we have by the functional calculus that operator $D^*(I+DD^*)^{-1/2}$ is bounded. Consequently we have
\begin{gather*}
DQ = I - (I+DD^*)^{-1}.
\end{gather*}
Writing $Q$ as \begin{gather*}
Q = D^*(I+DD^*)^{-1/2}(I+DD^*)^{-1/2},
\end{gather*}
we see that $Q$ is compact and so $D$ has compact right parametrix. Similar argument shows that~$Q$ is also a left parametrix.

Conversely, let $Q$ be a compact parametrix of $D$, i.e., $DQ = I + K_1$ and $QD = I + K_2$, where~$K_1$ and~$K_2$ are compact. Then consider
\begin{gather*}
(I+D^*D)^{-1/2} = (QD -K_2)(I+D^*D)^{-1/2} = QD(I+D^*D)^{-1/2} - K_2(I+D^*D)^{-1/2}.
\end{gather*}
Since $D(I+D^*D)^{-1/2}$ and $(I+D^*D)^{-1/2}$ are bounded and~$Q$ and $K_2$ are compact, it follows that the right hand side of the above equation is compact. A similar decomposition works for showing the compactness of $(I+DD^*)^{-1/2}$, thus completing the proof.
\end{proof}

\begin{cor}If $D$ is a Fredholm operator with compact parametrices, then $D^*D$ and $DD^*$ are Fredholm operators with compact parametrices.
\end{cor}
\begin{proof}Notice that $(I+D^*D)^{-1}$ and $(I+DD^*)^{-1}$ are resolvents of $D^*D$ and $DD^*$ respectively, and they are compact by the previous proposition.
\end{proof}

Operators with compact parametrices have the following simple stability property.

\begin{prop}Suppose $D$ is a Fredholm operator with compact parametrices. If~$a$ is a~bounded operator, then $D+a$ is Fredholm with compact parametrices.
\end{prop}

\begin{proof}If $Q$ is a compact parametrix of $D$ then it is also a parametrix of $D+a$.
\end{proof}

We have the following ``sandwich property" for operators with compact parametrices.
\begin{prop}\label{sandwich}
Let $D_i$ be closed operators for $i=1,2,3$, such that $D_1\subset D_2 \subset D_3$. If $D_1$ and $D_3$ are Fredholm operators and $D_2$ has compact parametrices, then both $D_1$ and $D_3$ have compact parametrices.
\end{prop}

\begin{proof}Since $D_2$ has compact parametrices, there exists a compact $Q$ such that $QD_2 = I + K$ for some compact operator~$K$. Since $D_1\subset D_2$ we have $\operatorname{dom}(D_1)\subseteq\operatorname{dom}(D_2)$ and therefore $QD_1 = I + K$. Since $D_1$ is Fredholm it has both left and right parametrices. The above shows that~$D_1$ has a~compact left parametrix. Consequently the right parametrix of $D_1$ must also be compact. A similar argument works for~$D_3$. This completes the proof.
\end{proof}

Next we turn our attention to spectral properties of operators with compact parametrices. As an example consider operators $D_1$, $D_2$, and $D_3$ all equal to $\frac{1}{i}\frac{d}{dx}$ on absolutely continuous functions in~$L^2[0,1]$ but with dif\/ferent boundary conditions: no boundary conditions for $D_1$, $f(0)=0$ for $D_2$, and periodic boundary conditions for~$D_3$. Then the spectrum $\sigma(D_1)$ of $D_1$ is all of~$\C$, $\sigma(D_2)$ is empty, and $D_3$ has a purely point spectrum. They are all Fredholm operators with compact parametrices, $(I+D_iD_i^*)^{-1/2}$ and $(I+D_i^*D_i)^{-1/2}$ are compact since~$D_iD_i^*$ and~$D_i^*D_i$ are Laplace operators on~$L^2[0,1]$ with elliptic boundary conditions.

\begin{prop} Let $D$ be a closed operator with compact parametrices. There are exactly three possibilities for the spectrum of~$D$:
$1)$~$\sigma(D) = \C$, $2)$~$\sigma(D) = \varnothing$, $3)$~$\sigma(D) = \sigma_p(D)$, the point spectrum of~$D$. In the last case, either~$\sigma(D)$ is finite or countably infinite with eigenvalues going to infinity.
\end{prop}

\begin{proof}
The examples above demonstrate all three possibilities. Suppose $\sigma(D)\neq \C$, then there exists a $\lambda_0$ such that $R_D(\lambda_0)$ exists. Since $D$ has compact parametrices, $R_D(\lambda_0)$ is compact. By spectral theory of compact operators we have
\begin{gather*}
\sigma(R_D(\lambda_0)) = \{0\}\cup\sigma_p(R_D(\lambda_0))
\end{gather*}
with three possibilities for the point spectrum: it's empty, f\/inite or countably inf\/inite tending to zero.

By assumption on $\lambda_0$ we have $0\notin\sigma(D-\lambda_0I)$. We claim that the mapping $\sigma_p(R_D(\lambda_0))\ni\lambda\mapsto \lambda^{-1}\in\sigma(D-\lambda_0I)$ is a bijection. Consider the following identity
\begin{gather*}
R_D(\lambda_0) - \lambda I = (D-\lambda_0I)^{-1} - \lambda I = -\lambda(D-\lambda_0 I)^{-1}\left((D-\lambda_0I) - \frac{1}{\lambda}I\right) .
\end{gather*}
If $\lambda$ is an eigenvalue for $R_D(\lambda_0)$ then it's clear that $\lambda^{-1}\in\sigma(D-\lambda_0I)$. Now suppose $0\neq \lambda\notin\sigma_p(R_D(\lambda_0))$, then since $R_D(\lambda_0)$ is compact, $R_D(\lambda_0) - \lambda I$ is invertible by the Fredholm alternative. Then we have the following
\begin{gather*}
\left((D-\lambda_0I) - \frac{1}{\lambda}I\right)^{-1} = -\lambda(R_D(\lambda_0) - \lambda I)^{-1}R_D(\lambda_0),
\end{gather*}
and the right-hand side is a bounded operator, which establishes the claim.

Using the bijection we can get all the information about the spectrum of $D$, since we have $\sigma(D-\lambda_0I) = \sigma(D)-\lambda_0$. If $\sigma_p(R_D(\lambda_0)) = \varnothing$ then we get that $\sigma(D) = \varnothing$, if $\sigma_p(R_D(\lambda_0))$ is f\/inite then $\sigma(D)=\sigma_p(D)$ and is f\/inite, and f\/inally if $\sigma_p(R_D(\lambda_0))$ is countably inf\/inite with eigenvalues tending to zero, then $\sigma(D) = \sigma_p(D)$ is countably inf\/inite with eigenvalues going to inf\/inity. This completes the proof.
\end{proof}

The last topic covered in this appendix is an analysis of operators of the form
\begin{gather*}
{\mathcal D} = \left[
\begin{matrix}
0 & D \\
D^* & 0
\end{matrix}\right],
\end{gather*}
which appear in the def\/inition of an even spectral triple.

\begin{prop}The operator ${\mathcal D}$ has compact parametrices if and only if the operator $D$ has compact parametrices.
\end{prop}
\begin{proof} If $Q$ is a compact parametrix of~$D$, let~$K_1$ and~$K_2$ be the compact operators such that $DQ = I + K_1$ and $QD = I + K_2$. Using $Q$ we can construct a parametrix of~${\mathcal D}$ by
\begin{gather*}
{\mathcal D}\left[
\begin{matrix}
0 & Q^* \\
Q & 0
\end{matrix}\right] = \left[
\begin{matrix}
I + K_1 & 0 \\
0 & I + K_2^*
\end{matrix}\right],
\end{gather*}
and similarly for the multiplication in the reverse order. These imply that ${\mathcal D}$ has compact parametrices. Conversely, if ${\mathcal D}$ has compact parametrices, its resolvent is compact. We can write down the resolvent for imaginary $-i\lambda$ as follows
\begin{gather*}
({\mathcal D} +i\lambda I)^{-1} = \left[
\begin{matrix}
i\lambda I & D \\
D^* & i\lambda I
\end{matrix}\right]^{-1} = \left[
\begin{matrix}
-i\lambda(\lambda^2I + DD^*)^{-1} & D(\lambda^2I + D^*D)^{-1} \\
D^*(\lambda^2I + DD^*)^{-1} & -i\lambda(\lambda^2I + D^*D)^{-1}
\end{matrix}\right].
\end{gather*}
Inspecting the diagonal elements of the above matrix we see that $D$ has compact parametrices by Proposition~\ref{ddstarprop}.
\end{proof}

\subsection*{Acknowledgements}
 We would like to thank the anonymous referees for relevant remarks to improve the paper.

\pdfbookmark[1]{References}{ref}
\LastPageEnding

\end{document}